\documentclass[letterpaper,11pt,reqno]{amsart}
\usepackage[utf8]{inputenc}
\usepackage{amsmath}
\usepackage{amsthm}
\usepackage{amssymb}
\usepackage{amsfonts,mathrsfs}
\usepackage{stmaryrd}
\usepackage{amsxtra}  
\usepackage{epsfig}
\usepackage{verbatim} 
\usepackage{enumerate}
\usepackage{xcolor}
\usepackage{enumitem}
\usepackage{bigints}
\usepackage{mathrsfs}
\usepackage{tikz-cd}
\usepackage{bm,bbm}
\usepackage{hyperref}
\usepackage[all]{xy}
\usepackage{mathtools, hyperref}
\usepackage{tikz}
\usepackage{tikz-cd}
\usetikzlibrary{shapes}
\usetikzlibrary{plotmarks}
\usepackage{booktabs}
\usepackage{siunitx}

\theoremstyle{plain}
\newtheorem{theorem}[equation]{Theorem}
\newtheorem{proposition}[equation]{Proposition}
\newtheorem{lemma}[equation]{Lemma}
\newtheorem{corollary}[equation]{Corollary}
\theoremstyle{definition}
\newtheorem{definition}[equation]{Definition}
\newtheorem{conjecture}[equation]{Conjecture}

\theoremstyle{remark}
\theoremstyle{remark}
\newtheorem{remark}[equation]{Remark}
\numberwithin{equation}{section}

\newcommand{\norm}[1]{\left\Vert#1\right\Vert}
\newcommand{\ol}{\overline}

\newcommand{\wt}{\widetilde}

\newcommand{\xdownarrow}[1]{%
  {\left\downarrow\vbox to #1{}\right.\kern-\nulldelimiterspace}
}

\newcommand{\cs}{{\mathcal S}}

\newcommand{\C}{{\mathbb C}}
\newcommand{\D}{{\mathbb D}}

\newcommand{\h}{{\mathbb H}}

\newcommand{\Q}{{\mathbb Q}}
\newcommand{\R}{{\mathbb R}}

\newcommand{\Z}{{\mathbb Z}}

\begin{document}

\title[Arithmetic properties and zeros of the Bergman kernel]{Arithmetic properties and zeros of the Bergman kernel on a class of quotient domains}

\author{Luke D. Edholm and Vikram T. Mathew}

\begin{abstract}
An effective formula for the Bergman kernel on $\h_{\gamma} = \{|z_1|^\gamma < |z_2| < 1 \}$ is obtained for rational $\gamma = \frac{m}{n} >1$.
The formula depends on arithmetic properties of $\gamma$, which uncovers new symmetries and clarifies previous results.
The formulas are then used to study the Lu Qi-Keng problem. 
We produce sequences of rationals $\gamma_j \searrow 1$, where each $\h_{\gamma_j}$ has a Bergman kernel with zeros (while $\h_1$ is known to have a zero-free kernel), resolving an open question on this domain class. 
\end{abstract}

\thanks{The first author was supported in part by Austrian Science Fund (FWF) grants: DOI 10.55776/I4557 and DOI 10.55776/P36884.}
\address{Department of Mathematics\\Universit\"at Wien, Vienna, Austria}
\email{luke.david.edholm@univie.ac.at}
\address{Department of Industrial Engineering and Operations Research\\Columbia University, New York, New York}
\email{vtm2111@columbia.edu}

\maketitle

\section{Introduction}

Given a domain (an open connected set) $\Omega \subset \C^n$, the Bergman space $A^2(\Omega)$ is the subspace of $L^2(\Omega)$ consisting of holomorphic functions.
The Bergman space is itself a Hilbert space, so there is a canonical orthogonal projector $P_\Omega: L^2(\Omega) \to A^2(\Omega)$, called the Bergman projection.
This projection is carried out by integrating against the Bergman kernel, the reproducing kernel function $K_\Omega : \Omega \times \Omega \to \C$ of $A^2(\Omega)$:
\begin{equation}\label{E:def-of;Bergman-proj}
P_\Omega f(z) = \int_\Omega K_\Omega(z,w) f(w)\,dV(w), \qquad f\in L^2(\Omega).
\end{equation}
The Bergman kernel is uniquely determined by the following three properties: for each fixed $w \in \Omega$, the function $K_\Omega(\cdot,w) \in A^2(\Omega)$; it is conjugate symmetric, i.e. $K_\Omega(z,w) = \ol{K_\Omega(w,z)}$; and it reproduces holomorphic functions, i.e. for $f \in A^2(\Omega)$, $P_\Omega f=f$.
In addition, $K_\Omega$ satisfies many useful extremal and invariance properties; see e.g. \cite{bergmanbook,Krantz_scv_book,JarPflBook08,Krantz_BergmanBook1_2013}.
In Section \ref{SS:BergKernProps} we recall those properties relevant to the present article.

For $\gamma>0$, define a family of bounded, non-smooth, weakly pseudoconvex domains called generalized Hartogs triangles by setting
\begin{equation}\label{E:def=gen-Hartogs}
\h_\gamma = \{(z_1,z_2) \in \C^2 : |z_1|^\gamma < |z_2| < 1 \}.
\end{equation}
When $\gamma$ is rational, $\h_\gamma$ can be realized as 
a quotient of $\D \times \D^*$ induced by a monomial mapping (see Section \ref{SS:Monomial-maps}).
Many aspects of Bergman theory on these domains have been studied in recent years; see \cite{chakzeytuncu,Edh16,EdhMcN16,EdhMcN16b,EdhMcN20,ChEdMc19,KhLiTh19,CKMM20,bcem,HuoWick2020a, HuoWick2020b,ChrKoe23} and the references therein.
The purpose of this paper is two fold: 
(1) to give a full and effective formula for Bergman kernel of $\h_\gamma$ for every rational $\gamma > 1$, and (2) to use the formulas to exhibit new symmetries and locate zeros of the kernel in $\h_\gamma \times \h_\gamma$.

The question on the existence/non-existence of zeros of Bergman kernel (the Lu Qi-Keng problem) goes back several decades and was originally motivated by the global well-definedness of Bergman representative coordinates; see e.g., \cite{LQK-66}, and the survey \cite{Boas98}.

A domain $\Omega$ is said to be a Lu Qi-Keng domain when its Bergman kernel is non-vanishing.
Appropriately interpreted, the zero set of the Bergman kernel is an analytic variety on $\Omega \times \Omega$, and transformation law \eqref{E:Berg-biholo-trans-law} shows this variety is a biholomorphic invariant of $\Omega$.
Simple domains such as the ball and polydisc have non-vanishing Bergman kernel, and it is an interesting open problem to ask when the taking of quotients of these and other Lu-Qi Keng domains induces zeros in the quotient Bergman kernels.
This has been investigated in special cases (certain ball quotients \cite{BFS99}, symmetrized polydiscs \cite{EdiZwo05,NikZwo06}, the tetrablock \cite{Try13} and quotients of higher dimensional Cartan type IV domains \cite{GhoZwo25}), but a unified approach to this problem is only in its early stages.

Since the zero set of the Bergman kernel is a biholomorphic invariant, a preliminary step to the Lu Qi-Keng problem on the domain family $\{\h_\gamma : \gamma >0 \}$ would be to identify its biholomorphic equivalence classes.
In Theorem \ref{T:gen-Hartogs-biholo} and Corollary \ref{C:biholo-equiv-classes} below, we prove that for distinct $\gamma,\eta \ge 1$, the domains $\h_\gamma$ and $\h_\eta$ are {\em never} biholomorphic.
On the other hand, given $\gamma \in (0,1)$, there is a unique $\wt\gamma \ge 1$ such that $\h_\gamma$ is biholomorphic to $\h_{\wt\gamma}$, and therefore the problem can be reduced to understanding the situation for $\gamma \ge 1$.
Up to the time of writing of this paper, the Lu Qi-Keng problem on $\h_\gamma$ has stood as follows:
\begin{itemize}
\item[$(a)$] The Bergman kernel on $\h_1$ (the classical Hartogs triangle) is non-vanishing; this can be seen directly from formula \eqref{E:BergKernel-HartogsTri} below.
\item[$(b)$] The Bergman kernel on $\h_\gamma$ has zeros for $\gamma \in [2,\infty)$. 
This was shown first  for integers $\gamma \ge 2$ in \cite{Edh16}, then for all other $\gamma>2$ in \cite{EdhMcN16b}.
\item[$(c)$] It has remained open as to whether the Bergman kernel of $\h_\gamma$ has zeros for $\gamma \in (1,2)$.
\end{itemize} 

In \cite{EdhMcN16b}, when Edholm and McNeal found zeros of $\h_\gamma$ for $\gamma \ge 2$, it was also observed that the same approach fails to address $\gamma \in (1,2)$; see Section \ref{SS:previous-results} below for more details. 
The question has remained open since:
\begin{itemize}
\item[{\bf (Q1)}] {\em For which $\gamma \in (1,2)$ does the Bergman kernel have zeros?} \label{Q:which-gamma-in-(1,2)?}
\end{itemize}

A weaker version of (Q1) has also been raised; see \cite{Edh-Thesis}.
Since the topologically equivalent domains $\h_{\gamma_j}$ tend to $\h_1$ as $\gamma_j \searrow 1$, Ramadanov's theorem (more precisely, an extended version of the more classical result; see \cite{Kra06}) shows that the Bergman kernels $K_{\gamma_j}(\cdot,\cdot)$ of $\h_{\gamma_j}$ tend to the Bergman kernel $K_1(\cdot,\cdot)$ of $\h_1$ locally normally on $\h_1 \times \h_1$.
As $K_1$ is known to be non-vanishing, it is interesting to ask
\begin{itemize}
\item[{\bf (Q2)}] {\em Is there some $\epsilon > 0$ so that $K_\gamma(\cdot,\cdot)$ is non-vanishing on $\h_\gamma \times \h_\gamma$ for $\gamma \in [1,1+\epsilon)$?} \label{Q:wiggle-room-non-vanish}
\end{itemize}

This paper addresses both questions.
In Section \ref{S:zeros-of-the-Bergman-kernel} we exhibit infinitely many rational $\gamma \in (1,2)$ for which $K_\gamma(\cdot,\cdot)$ has zeros inside $\h_\gamma \times \h_\gamma$, marking the first progress made on (Q1).
These $\gamma$ values accumulate at 1, which answers (Q2) by proving no such $\epsilon$ can exist.

\subsection{An effective kernel formula}

While the existence of the Bergman kernel is guaranteed by abstract Hilbert space theory, concrete information about its properties as a function $K_\Omega : \Omega \times \Omega \to \C$ (e.g. its asymptotics near the boundary, whether it has zeros inside $\Omega \times \Omega$, and if so, properties of the zero set) requires hard analysis.

Our first main result is a new and effective formula for the Bergman kernel of $\h_{m/n}$, valid for any rational $\frac{m}{n} > 1$.
Previous work (see \cite{EdhMcN16b,CKMM20,Almugh23}) has expressed these kernels (which have been known since \cite{EdhMcN16b} to be rational functions) in one form or another, but until now a description of the numerator polynomial from which it is possible to study the zero set of the kernel when $\gamma \in (1,2)$ has remained elusive.
For comparison, these past results are recalled in Section \ref{SS:previous-results}.

\begin{theorem}\label{T:Bergman-kernel-formula}
Let $m,n \in \Z^+$ with $m>n$ and $\gcd(m,n)=1$. 
Given points $z=(z_1,z_2)$ and $w=(w_1,w_2)$ in $\h_{m/n}$, let $s = z_1\ol{w}_1$, $t = z_2\ol{w}_2$.
The Bergman kernel of $\h_{m/n}$ is
\begin{subequations}
\begin{equation}\label{E:Bergman-kernel-formula-intro}
K_{m/n}(z,w) = \frac{P_{m,n}(s,t)}{m \pi^2(1-t)^2(t^n-s^m)^2},
\end{equation}
where the numerator is a polynomial of degree $2m-1$, given explicitly by
\begin{equation}\label{E:Pmn-as-sum-of-5-terms}
P_{m,n}(s,t) = p_0(s,t) + p_1(s,t) + p_2(s,t) + p_3(s,t) + p_4(s,t),
\end{equation}
where
{\allowdisplaybreaks
\begin{align}
p_0(s,t) &= m^2 s^{m-1}t^n, \label{E:def-p0-intro}\\
p_1(s,t) &= \sum_{j=0}^{m-2} (j+1)(\kappa(j)+1) s^j t^{2n-L(j)}, \label{E:def-p1-intro} \\
p_2(s,t) &= \sum_{j=0}^{m-2} (j+1)(m-\kappa(j)-1) s^j t^{2n+1-L(j)}, \label{E:def-p2-intro} \\
p_3(s,t) &= \sum_{j=0}^{m-2} (m-j-1)(\kappa(j)+1) s^{j+m} t^{n-L(j)}, \label{E:def-p3-intro} \\
p_4(s,t) &= \sum_{j=0}^{m-2} (m-j-1)(m-\kappa(j)-1) s^{j+m} t^{n+1-L(j)}, \label{E:def-p4-intro}
\end{align}
}
and the expressions $L$ and $\kappa$ are given by
\begin{align}
L(j) = \left\lceil \frac{1+n(j+1)}{m} \right\rceil, \qquad 
\kappa(j) = m + n - 1 +nj - mL(j). \label{E:def-of-mn-ceiling-func-kappaj-intro}
\end{align}
\end{subequations}
\end{theorem}

Theorem \ref{T:Bergman-kernel-formula} is proved in Section \ref{S:effective-formula}, where modular arithmetic is used to solve several linear Diophantine equations depending on $m$ and $n$, thus determining the complete set of monomials appearing in the numerator polynomial (Proposition \ref{P:Finding-effective-coefficients}).
This information is then combined with functional identities for $L$ and $\kappa$ (Lemma \ref{L:props-of-L-and-kappa}) to produce the kernel.

The new formulas allow for effective analysis and the discovery of new symmetries by improving (in fact optimizing) the presentation of the numerator polynomial.
For example, the most recent of the previous formulas for $K_{m/n}(z,w)$  (\cite[Proposition 2]{Almugh23}; see Theorem \ref{T:Almugh-formula} below) requires the computation of $(2m-1)(2n+1)$ monomial coefficients in the numerator.
Theorem \ref{T:Bergman-kernel-formula}, however, shows that all but exactly $4m-3$ of these coefficients are equal to zero.
For large $n$ this is an enormous simplification on both a computational and a theoretical level, one that leads to the discovery of new symmetries that can be used to study the zero set of the kernel.


\subsection{Monomial maps; symmetries and zeros of the Bergman kernel}\label{SS:Monomial-maps}

Given a relatively prime pair $m,n \in \Z^+$, we recall the point of view taken in \cite{Edh16,bcem,CKMM20,Almugh23} in which $\h_{m/n}$ is realized as a quotient domain.
The image of $\D \times \D^*$ under the monomial map 
\begin{equation*}
\Phi_{m,n}(z_1,z_2) = (z_1 z_2^n,z_2^m)
\end{equation*}
can be identified with a rational Hartogs triangle
\begin{equation*}
\Phi_{m,n}(\D \times \D^*) = \h_{m/n}.
\end{equation*}

More precisely, $\Phi_{m,n}: \D \times \D^* \to \h_{m/n}$ is an $m$-sheeted proper holomorphic branched covering.
Using the language of \cite{bcem,ChakEdh24}, $\Phi_{m,n}$ is said to be a proper holomorphic map of quotient type with respect to a group $\Gamma_{m,n} \subset \mathrm{Aut}(\D \times \D^*)$, the deck transformation group of $\Phi_{m,n}$.
In this case $\Gamma_{m,n}$ is cyclic of order $m$.
In light of these considerations, the complexity of the formulas in Theorem \ref{T:Bergman-kernel-formula} demonstrates that an intricate Bergman theory can arise on quotient domains, even when the starting domains and maps are quite simple.
In our setting $\Phi_{m,n}$ forces arithmetic properties of the rational index $\gamma = \frac{m}{n}$ into the Bergman kernel via Bell's transformation rule for proper maps (see \cite{belltransactions}).

\begin{figure}
\centering
\begin{tikzpicture}[x={(0,1cm)},y={(1 cm,0)}]
\draw[-{latex}, thick] (0,0) -- (3.5,0) node[anchor=east] {$|t|$};
\draw[-{latex}, thick] (0,0) -- (0,3.5) node[anchor=west] {$|s|$};
\shadedraw [gray, domain=0:3] (0,0) -- plot  (1/3*\x*\x,\x) -- (3,0);
\draw[-, very thick,dashed] (0,0) -- (1.8,1.8) node[anchor=east] {$\{s = t\}\,\,$};
\draw[-, very thick,dashed] (1.8,1.8) -- (3,3);
\end{tikzpicture} 
\caption{The punctured disc $\Delta^*$ lying obliquely inside $\h_\gamma$.}
\label{Fig:punctured-disc-Delta*}
\end{figure}
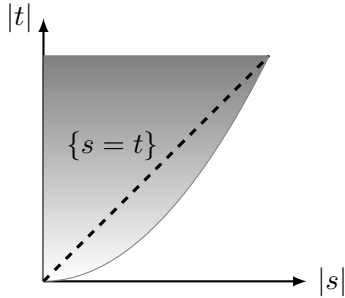

\subsubsection{Palindromic polynomials}\label{SSS:palidromic-polys}
The numerator polynomial of the Bergman kernel of $\h_{m/n}$ satisfies an interesting symmetry that assists our study of the Lu Qi-Keng problem.

Given points $z,w \in \h_\gamma$, we continue the notation of Theorem \ref{T:Bergman-kernel-formula} and write $s = z_1\ol{w}_1$, $t = z_2\ol{w}_2$.
It is easy to check via Corollary \ref{C:comm-diag} and Proposition \ref{P:gen-Hartogs-are-their-own-Reinhardt-powers} that the point $(s,t)$ is also contained in $\h_\gamma$, and that every point of $\h_\gamma$ can be realized in this way.

Now think of $(s,t)$ as coordinates on $\C^2$, and consider the complex line $\{s = t \}$.
For $\gamma>1$, this line intersects the interior of $\h_\gamma = \{|s|^\gamma < |t| < 1 \}$.
The intersection is a punctured disc $\Delta^*$ lying obliquely inside of $\h_\gamma$ (see Figure \ref{Fig:punctured-disc-Delta*}):
\begin{align*}
\Delta^* 
= \{(s,s) \in \h_\gamma : s \in \D^* \}.
\end{align*}

Now 
restrict the numerator of \eqref{E:Bergman-kernel-formula-intro} to the line $\{s=t\}$.
The one variable polynomial $s \mapsto P_{m,n}(s,s)$ can be shown to have a root of order $2n-1$ at the origin, which suggests the following definition 
\begin{align*}
Q_{m,n}(s) := s^{1-2n} P_{m,n}(s,s).
\end{align*}
Thus, roots of the one variable polynomial $Q_{m,n}(s)$ inside $\D^* \subset \C$ are in one-to-one correspondence with the roots of $P_{m,n}(s,t)$ restricted to $\Delta^* \subset \h_\gamma$.

Understanding these roots is made easier due to the following unexpected symmetry in the coefficients of $Q_{m,n}$:

\begin{theorem}\label{T:Q-palindromic-intro}
For any pair of relatively prime integers $m>n$, $Q_{m,n}$ is a palindromic polynomial of degree $2m-2n$, i.e.,
\begin{equation}\label{E:Q-palindromic-intro}
Q_{m,n}(s) = s^{2m-2n} Q_{m,n}(s^{-1}).
\end{equation}
\end{theorem}
Theorem \ref{T:Q-palindromic-intro} is obtained as part of Proposition \ref{P:Qmn-palindromic}, which establishes additional identities analogous to \eqref{E:Q-palindromic-intro} on pieces of the full polynomial.
The palindromic property of $Q_{m,n}$ allows for indirect study of roots inside $\D^*$ by examining the polynomial's behavior on the unit circle.

This circle of ideas pays off in Section \ref{S:zeros-of-the-Bergman-kernel} where it is shown that
$Q_{m,n}$ is non-vanishing on the unit circle for all relatively prime pairs $m>n$ such that $m-n = 1$ or $2$:

\begin{theorem}\label{T:half-zeros-inside-intro}
Let $m>n$ be a pair of relatively prime integers.
If $m-n = 1$ or $2$, then $Q_{m,n}$ has exactly half of its roots inside the unit disc.
Consequently, the Bergman kernel of $\h_{m/n}$ has zeros inside $\h_{m/n} \times \h_{m/n}$.
\end{theorem}

Proofs of different pieces of this theorem are found in Sections \ref{SS:k=1} and \ref{SS:k=2}.
As a consequence of this result, we obtain infinitely many rationals accumulating at $\gamma = 1$ such that the Bergman kernel has zeros, resolving (Q2) above with a negative answer.

\begin{corollary}
Despite the fact that the Bergman kernel $K_1(z,w)$ of $\h_1$ is non-vanishing on $\h_1 \times \h_1$, there is no $\epsilon > 0$ for which the Bergman kernels $K_\gamma(z,w)$ of $\h_\gamma$ are non-vanishing on $\h_\gamma \times \h_\gamma$ for $\gamma \in [1,1+\epsilon)$.
\end{corollary}

\subsection{Past results}\label{SS:previous-results}

The Bergman kernel of $\h_{\gamma}$ was first computed for $\gamma \in \Z^+$ by Edholm in \cite{Edh16}, then for $\gamma = \frac{m}{n} \in \Q^+$ by Edholm and McNeal in \cite{EdhMcN16b}, where it was expressed as a sum of $m$ subkernels, each being an explicit rational function.
Sharp estimates on these subkernels imply a sharp estimate on the full Bergman kernel, which can then be used to characterize the $L^p$-mapping properties of the Bergman projection based on the rationality/irrationality of $\gamma$; see \cite[Theorem 1.1 and 1.2]{EdhMcN16b}.

\subsubsection{Known results on the vanishing set}
Because the $L^p$ results in \cite{EdhMcN16b} were deducible directly from the kernel estimates, little attempt was made to further analyze the numerator of the Bergman kernel, or to re-write it in a way that could make symmetries like Theorem \ref{T:Q-palindromic-intro} easier to detect.
The only result in this direction was \cite[Theorem 3.3]{EdhMcN16b}, where it was noted that all but one of the subkernels vanish identically on the variety $\{s=0\}$, meaning the full Bergman kernel restricted to this variety is given by a particularly simple formula:
\begin{equation*}
K_\gamma((0,z_2),(0,w_2)) = \frac{1+(\gamma-1)z_2\ol{w}_2}{\gamma \pi^2 z_2\ol{w}_2(1-z_2\ol{w}_2)^2}.
\end{equation*}
(This formula is valid for all $\gamma>0$, including irrationals.)

When $\gamma>2$, it is possible by inspection to find $z_2,w_2$ with $0< |z_2|, |w_2| < 1$ for which the Bergman kernel vanishes. 
Here, it is crucial that $z_2,w_2 \in \D^*$ so that both $(0,z_2),(0,w_2) \in \h_\gamma$.
On the other hand, for $0 < \gamma \le 2$, it is not possible to find such $z_2,w_2$, meaning that those Bergman kernels are non-vanishing on $\{s=0\}$.
The $\gamma = 2$ case had previously been addressed by Edholm in \cite[Corollary 4.2]{Edh16}, where an explicit zero of $K_2$ inside $\h_2 \times \h_2$ (but away from $\{s=0\}$) was produced.

This approach of looking at the kernel restricted to a variety inspired the restriction of the Bergman kernel to $\{s=t\}$ considered in the present paper.

\subsubsection{The Bergman kernel in terms of the function $D_m$}
A different approach to the Bergman kernels of $\h_{m/n}$ comes from Chakrabarti et al. in \cite{CKMM20}, where they obtained formulas as special cases of calculations on so-called Signature 1 domains in $\C^n$.
In this work, proper holomorphic maps are used to cover the domains under consideration by simpler product domains for which the Bergman kernel was already known.
Bell's transformation law (see \cite{belltransactions}) is then applied and the kernel on the target is expressed as a linear combination of functions closely related to the Bergman kernel on the source. 

The difficulty with this approach is parsing this linear combination into a digestible formula that can then be used to effectively study some other problem (such as the mapping properties of the Bergman projection or, especially, locating zeros of the Bergman kernel).

One of the main observations of \cite{CKMM20} is that the Bergman kernels of any Signature 1 domain are rational functions whose numerators can be succinctly written using the following function defined on the integers:
\begin{align}\label{E:Dbasic_def}
D_m(\beta) 
= \begin{cases}
\beta + 1, & 0 \leq \beta \leq m - 1, \\
2m - 1 - \beta, & m \leq \beta \leq 2m - 2, \\
0, & \mathrm{otherwise}.
\end{cases}
\end{align}

This function arises out of combinatorial considerations as it describes the coefficients on the right hand side of the following polynomial equation:
\begin{equation*}
\left( \frac{1-s^m}{1-s} \right)^2 = \sum_{\beta=0}^{2m-2} D_m(\beta) s^\beta.
\end{equation*}

Following the initial work (\cite{CKMM20}) expressing the Bergman kernels on Signature 1 domains in terms of $D_m$, the formulas were substantially simplified in the two dimensional case in a recent paper by Almughrabi \cite{Almugh23}:

\begin{proposition}[\cite{Almugh23}, Proposition 2]\label{T:Almugh-formula}
Let $m,n \in \Z^+$ with $m > n$ and $\gcd(m,n)=1$.
Letting $s = z_1\ol{w}_1$, $t = z_2\ol{w}_2$, the Bergman kernel of $\h_{m/n}$ is given by 
\begin{equation}\label{E:Almugh-formula}
K_{m/n}(z,w) = \sum_{\beta \in \Z^2} \frac{D_m(\beta_1) D_m(m\beta_2 + n\beta_1+m+n-1-2mn)s^{\beta_1}t^{\beta_2}}{\pi^2 m (t^n - s^m)^2(1-t)^2}.
\end{equation}
In this formula, the terms in the numerator polynomial are only possibly non-zero when 
\begin{equation}\label{E:two-ineqs}
0\le\beta_1\le 2m-2, \qquad \mathrm{and} \qquad 0\le\beta_2\le 2n.
\end{equation}
\end{proposition}

The presence of $D_m$ in the numerator polynomial provides some insight into the algebra and combinatorics injected into the problem by the covering maps $\Phi_{m,n} : \D \times \D^* \to \h_{m,n}$.
Unfortunately, leaving the numerator in the form of \eqref{E:Almugh-formula} obfuscates basic information about the polynomial: its degree, which monomials appear in the sum with non-zero coefficients, and what these coefficients are.

The kernel formulas given in Theorem \ref{T:Bergman-kernel-formula} entirely resolve these uncertainties.
The conditions on $(\beta_1,\beta_2)$ in \eqref{E:two-ineqs} requires us to compute $(2m-1)(2n+1)$ polynomial coefficients, while polynomial formulas \eqref{E:def-p0-intro} through \eqref{E:def-p4-intro} show that all but exactly $4m-3$ of these coefficients are non-zero.

\subsection{Acknowledgments}
The authors thank Bernhard Lamel for many discussions on the intersection of discrete mathematics and Bergman theory, as well as the anonymous referee for providing positive feedback and useful suggestions.

This paper grew out of an undergraduate research project that Mathew did with Edholm while they were both at the University of Michigan in 2019-20.
After a COVID-induced, years long hiatus taking them to opposite corners of the globe, they are happy to have returned to the project to complete many partial results obtained years earlier.


\section{Preliminaries}\label{S:Bergman-kernel-Reinhardt-domains}

\subsection{Reinhardt domains}
A domain $\Omega \subset \C^n$ is called Reinhardt (with respect to the origin) when it admits a natural $n$-torus action allowing for independent rotation in each coordinate. 
In other words,
\begin{equation*}
(z_1,z_2,\cdots,z_n) \in \Omega \quad \Longleftrightarrow \quad (e^{i \theta_1} z_1, e^{i \theta_2} z_2,\cdots, e^{i \theta_n} z_n) \in \Omega
\end{equation*}
for any choice of $\theta_1,\dots,\theta_n \in \R$. 
In $\C^1$ the only Reinhardt domains are (1) discs, (2) the full plane, (3) annuli, (4) the plane with a closed disc removed, (5) once punctured discs with the center removed, and (6) the once punctured plane.
In these settings, each holomorphic function is represented by a {\em globally valid} Laurent expansion about the origin converging normally on compact subsets of the domain. 
Each of the above domain families has its own distinct flavor of holomorphic function theory and this distinction can be drastic, as the contrast between the classical theories of entire functions and holomorphic functions on the unit disc indicates.
 
Holomorphic functions on Reinhardt domains in $\C^n$ also admit globally valid Laurent expansions, and as before, two distinct Reinhardt domains can give rise to drastically different holomorphic function theories; e.g., function theory on the unit ball (see, e.g., \cite{Rudin-BallBook}) versus the unit polydisc (see, e.g., \cite{Rudin-PolydiscBook}).
Properties of the Bergman kernel (its boundary asymptotics, or a description of its zero set when it exists), and Bergman projection (its behavior acting in $L^p$, Sobolev, Hölder classes) give a way to gauge this.

A domain $\Omega \subset \C^n$ is called a domain of holomorphy when it is the domain of existence for some holomorphic function that cannot be analytically continued past any boundary point of $\Omega$.
The solution to the famous Levi problem (see \cite{Siu_LeviProblemSurvey,Krantz_scv_book}) characterizes the domains of holomorphy as the so-called pseudoconvex ones. 
Several equivalent geometric descriptions of pseudoconvexity exist in the general setting, and for Reinhardt domains it is particularly easy to characterize. 
We need two definitions:

Given a Reinhardt domain $\Omega \subset \C^n$, its {\em logarithmic shadow} is defined to be the set
\begin{equation*}
\mathrm{Log}(\Omega) := \{ (\log|z_1|,\cdots,\log|z_n|) \in \R^n : z \in \Omega \,\, \mathrm{with\,\,each} \,\, z_j \neq 0 \}.
\end{equation*}

A domain is called {\em weakly relatively complete} if, for each $j$ with $\Omega \cap \{z_j = 0\} \neq \varnothing$, that $(z',0,z'') \in \Omega$ whenever there exists some $(z',z_j,z'') \in \Omega \subset \C^{j-1} \times \C \times \C^{n-j}$.

\begin{proposition}[\cite{JarPflBook08}, Proposition 1.11.13]\label{P:Pseudoconvex-Reinhardt}
A Reinhardt domain $\Omega \subset \C^n$ is pseudoconvex if and only if $\mathrm{Log}(\Omega) \subset \R^n$ is a convex set and $\Omega$ is weakly relatively complete.
\end{proposition}

It is easily verified from Proposition \ref{P:Pseudoconvex-Reinhardt} that the domains $\h_\gamma$ are pseudoconvex.

\subsubsection{Biholomorphisms of Reinhardt domains}

Recall that domains $\Omega_1$ and $\Omega_2$ are biholomorphic if there is a bijective holomorphic mapping $\Omega_1 \to \Omega_2$ whose inverse is also holomorphic.
A long standing problem in several complex variables aims to understand the moduli space of biholomorphically equivalent domains.
This problem is extremely difficult, even in relatively tame settings (e.g. bounded complete pseudoconvex Reinhardt domains in $\C^n$, see \cite{RongSTYau09,RongGaoYau10}).

In the study of biholomorphisms between Reinhardt domains, a simple family of maps takes center stage:

\begin{definition}\label{D:def-of-algebraic-equiv}
An {\em algebraic automorphism} of $(\C^*)^n$ is a holomorphic automorphism
\begin{equation*}
\Phi = (\Phi_1,\cdots,\Phi_n): (\C^*)^n \to (\C^*)^n
\end{equation*}
with components of the form
\begin{equation}
\Phi_j(z) = \alpha_j z_1^{a_{1j}} z_2^{a_{2j}}\cdots z_n^{a_{nj}},
\end{equation}
where the matrix $A = (a_{ij}) \in \mathrm{GL}_n(\Z)$.
Two Reinhardt domains are said to be {\em algebraically equivalent} if there is a biholomorphic mapping between them induced by an algebraic automorphism of $(\C^*)^n$.
\end{definition}


\begin{proposition}[\cite{Shim88} Section 4;  \cite{Kruz88} Section 1]\label{P:biholo=algequiv}
If two bounded Reinhardt domains are biholomorphic, they are algebraically equivalent.
\end{proposition}

Proposition \ref{P:biholo=algequiv} leads to a complete determination of the biholomorphic equivalence classes in the family $\{\h_\gamma: \gamma>0 \}$; see Theorem \ref{T:gen-Hartogs-biholo} below.

\subsection{Bergman kernel properties}\label{SS:BergKernProps}

Let $U$ and $V$ be two domains in $\C^n$ and suppose there is a biholomorphism $\Phi:U \to V$. 
Then the Bergman kernel of $U$ has zeros if and only if the Bergman kernel of $V$ has zeros.
More precisely, if
\begin{equation*}
X = \left\{ (z,w) \in U \times U : K_U(z,w) = 0 \right\}, \qquad Y = \left\{ (z',w') \in V \times V : K_V(z',w') = 0 \right\},
\end{equation*}
then 
\begin{equation}\label{E:zero-set-transformation}
Y = \left\{ (\Phi(z),\Phi(w)) : (z,w)\in X \right\}.
\end{equation}
This is a direct consequence of the biholomorphic transformation law:
\begin{equation}\label{E:Berg-biholo-trans-law}
K_U(z,w) = \det\Phi'(z) \cdot K_V(\Phi(z),\Phi(w)) \cdot \ol{\det \Phi'(w)}.
\end{equation}
Indeed, since $\Phi$ is a biholomorphism, $\det\Phi'$ is necessarily non-vanishing (see \cite{Krantz_scv_book}).
Thus in the sense of \eqref{E:zero-set-transformation}, the zero set of the Bergman kernel is a biholomorphic invariant.

The map $(z_1,z_2) \mapsto \big(z_1{z_2}^{-1},z_2\big)$ is a biholomorphism from $\h_1$ to $\D \times \D^*$ (a domain with a non-vanishing Bergman kernel), so we can immediately conclude $K_{\h_1}(z,w)$ is non-vanishing inside $\h_1 \times \h_1$.
This can also be seen directly from the formula
\begin{equation}\label{E:BergKernel-HartogsTri}
K_{\h_1}(z,w) = \frac{z_2 \ol{w}_2}{\pi^2(1-z_2\ol{w}_2)^2 (z_2\ol{w}_2 - z_1\ol{w}_1)^2}.
\end{equation}

\subsubsection{Orthonormal bases}
Given a Reinhardt domain $\Omega \subset \C^n$ and a multi-index $\alpha \in \Z^n$ with corresponding monomial $e_\alpha(z) = z^\alpha$, we define the set of $L^2$-allowable indices to be
\[
\cs_2(\Omega) := \{ \alpha \in \Z^n : e_\alpha \in A^2(\Omega) \}.
\]
Since the square-integrable monomials form an orthogonal basis for the Bergman space on $\Omega$, we are able to write the Bergman kernel as
\begin{equation}\label{E:Berg-kern-infinite-sum}
K_\Omega(z,w) = \sum_{\alpha \in \cs_2(\Omega)} \frac{e_\alpha(z) \ol{e_\alpha(w)}}{\norm{e_\alpha}_2^2}
= \sum_{\alpha \in \cs_2(\Omega)} \frac{ (z_1 \ol{w}_1)^{\alpha_1} \cdots (z_n \ol{w}_n)^{\alpha_n} }{\norm{e_\alpha}_2^2},
\end{equation}
where the series converges locally normally on $\Omega \times \Omega$.

Observe that on Reinhardt domains where the boundary intersects the center of Reinhardt symmetry (like $\h_\gamma$), it is possible for certain negative powers to appear on square-integrable monomials.
But an elementary computation shows that any perturbation of $\gamma$ causes the set $\cs_2(\h_\gamma) \subset \Z^2$ to either lose or gain an infinite number of indices.
This fact underlies the drastic difference in Bergman kernel formulas of the domains $\h_{\gamma_1} \neq \h_{\gamma_2}$, even when $ |\gamma_1 - \gamma_2|$ is taken to be arbitrarily small.

\subsubsection{Factorization of the Bergman kernel}
Equation \eqref{E:Berg-kern-infinite-sum} shows the Bergman kernel on a Reinhardt domain can be written as a function of the variables $t_j := z_j\ol{w}_j$. 
In other words, if $\psi : \C^n \times \C^n \to \C^n$ is the map
\begin{equation*}
\psi(z,w) = (z_1\ol{w}_1,\cdots,z_n\ol{w}_n),
\end{equation*}
the Bergman kernel factors as $K_\Omega = k_\Omega \circ \psi$, where $k_\Omega$ is a function defined on the closely related Reinhardt domain
\begin{equation}\label{E:def-Omega-tilde}
\wt{\Omega} = \{t \in \C^n : \exists \, z,w \in \Omega \,\, \mathrm{with} \,\, t = (z_1\ol{w}_1,\cdots,z_n\ol{w}_n) \}.
\end{equation}

When $\Omega$ is also assumed to be pseudoconvex, $\wt\Omega$ admits a simpler description.
Following \cite{ChakEdh24}, for any $m>0$ we define the $m$-th Reinhardt power of $\Omega$ to be the Reinhardt domain
\begin{equation}\label{E:def-Reinhardt-square}
\Omega^{(m)} := \left\{ z \in \C^n : (|z_1|^{\frac{1}{m}},\cdots,|z_n|^{\frac{1}{m}}) \in \Omega \right\}.
\end{equation}

\begin{proposition}\label{P:Omega-tilde=Omega2}
Let $\Omega \subset \C^n$ be a pseudoconvex Reinhardt domain.
Then $\wt\Omega = \Omega^{(2)}$.
\end{proposition}
\begin{proof}
The containment $\Omega^{(2)} \subset \wt\Omega$ always holds.
Indeed, given $\zeta \in \Omega^{(2)}$, we have
\[
\zeta := (r_1 e^{i\theta_1}, \dots, r_n e^{i\theta_n}) \in \Omega^{(2)} 
\qquad \Longleftrightarrow \qquad
z := \left(\sqrt{r_1}, \dots, \sqrt{r_n} \right) \in \Omega.
\]
Now take $w = z$ and set $t := (z_1\ol{w}_1,\cdots,z_n\ol{w}_n) = (r_1,\cdots,r_n)$, which is contained in $\wt\Omega$ by definition \eqref{E:def-Omega-tilde}.
Since $\wt\Omega$ is Reinhardt, $\zeta \in \wt\Omega$.

The containment $\wt\Omega \subset \Omega^{(2)}$ requires pseudoconvexity.
Let $t = (t_1,\dots,t_n) \in \wt\Omega$ and choose
\[
z = (z_1,\dots,z_n)\in \Omega, \qquad w = (w_1,\dots,w_n) \in \Omega,
\]
such that $t_j = z_j \ol{w}_j$ for all $j$.
For now, assume that each $t_j \neq 0$.
Since $\Omega$ is pseudoconvex, the set $\mathrm{Log}(\Omega) \subset \R^n$ is convex and therefore the midpoint
\begin{equation}\label{E:midpoint-log-convex}
\left( \frac{\log|z_1|+\log|w_1|}{2} , \dots, \frac{\log|z_n|+\log|w_n|}{2}\right) \in \mathrm{Log}(\Omega).
\end{equation}
Upon noting that $\mathrm{Log}\big(\Omega^{(2)}\big) 
= \left\{ 2\eta : \eta \in \mathrm{Log}(\Omega) \right\}$, we see from \eqref{E:midpoint-log-convex} that
\[
(\log|t_1|,\dots,\log|t_n| ) = (\log|z_1| + \log|w_1|,\dots, \log|z_n| + \log|w_n| ) \in \mathrm{Log}\big(\Omega^{(2)}\big),
\]
which shows that $t = (t_1,\dots,t_n) \in \Omega^{(2)}$.

When some of the components $t_j$ of $t$ equal $0$, we take $z_j = w_j = 0$, then repeat the midpoint argument above on the logarithmic shadow of the lower dimensional Reinhardt domain obtained by intersecting $\Omega$ with the coordinate axes of the zeroed out components.
This lower dimensional domain is also pseudoconvex and its $m$-th Reinhardt powers are obtained by intersecting $\Omega^{(m)}$ with the same coordinate axes.
The rest of the argument goes through mutatis mutandis, confirming in this case that $t \in \Omega^{(2)}$.
\end{proof}


\begin{corollary}\label{C:comm-diag}
Given a pseudoconvex Reinhart domain  $\Omega \subset \C^n$, the map $\psi: \C^n \times \C^n \to \C^n$ restricts to a surjection of $\Omega \times \Omega$ onto $\Omega^{(2)}$.
Further, there is a holomorphic function $k_\Omega: \Omega^{(2)} \to \C$ so that the following diagram commutes
\begin{equation}\label{CD:General}
\begin{tikzcd}
\Omega \times \Omega \arrow{rr}{\psi} \arrow{rrd}[swap]{K_\Omega} && \Omega^{(2)} \arrow{d}{k_\Omega}\\
&& \C
\end{tikzcd}
\end{equation}
\end{corollary}
\begin{proof}
Since $\psi(\Omega \times \Omega) = \wt\Omega = \Omega^{(2)}$, the factorization of $K_\Omega = k_\Omega \circ \psi$ illustrated in \eqref{CD:General} holds.  
The function $k_\Omega$ is holomorphic on $\Omega^{(2)}$ because the power series 
\begin{equation}\label{E:holo-func-on-Omega2}
k_\Omega(t_1,\cdots,t_n) = \sum_{\alpha \in \cs_2(\Omega)} \frac{ t_1^{\alpha_1} \cdots t_n^{\alpha_n} }{\norm{e_\alpha}_2^2}
\end{equation}
converges locally normally on $\Omega^{(2)}$ -- this is a direct consequence of the local normal convergence of the series defining the Bergman kernel on $\Omega \times \Omega$ in \eqref{E:Berg-kern-infinite-sum}.
\end{proof}

\begin{definition}
Given a pseudoconvex Reinhardt domain $\Omega \subset \C^n$, we call the function $k_\Omega : \Omega^{(2)} \to \C$ the {\em holomorphic factor of the Bergman kernel}.
\end{definition}

\begin{corollary}\label{C:LQK-Reinhardt-shadow}
Let $\Omega$ be a pseudoconvex Reinhardt domain.
The zero set of the Bergman kernel $K_\Omega$ in $\Omega \times \Omega$ is the pullback via $\psi$ of the zero set of the holomorphic factor $k_\Omega$ restricted to $\Omega^{(2)}$.
In particular, $K_\Omega$ has zeros in $\Omega \times \Omega$ if and only if $k_\Omega$ has zeros in $\Omega^{(2)}$.
\end{corollary}

\subsection{Properties of $\h_\gamma$}

For later use, we record some observations about $\h_\gamma$.
First see that the generalized Hartogs triangles are  unchanged by taking Reinhardt powers.

\begin{proposition}\label{P:gen-Hartogs-are-their-own-Reinhardt-powers}
For $m>0$, $\gamma>0$ the generalized Hartogs triangle satisfies $\h_\gamma^{(m)} = \h_\gamma$, and so the holomorphic factor of the Bergman kernel is itself a function $k_\gamma : \h_\gamma \to \C$.
\end{proposition}
\begin{proof}
Given $m, \gamma > 0$, we have the following chain of equivalent statements:
\begin{align*}
z \in \h_{\gamma}^{(m)} \quad \Longleftrightarrow \quad \big(|z_1|^{\frac{1}{m}},|z_2|^{\frac{1}{m}}\big) \in \h_\gamma \quad &\Longleftrightarrow \quad |z_1|^{\frac{\gamma}{m}} < |z_2|^{\frac{1}{m}}  < 1 \\
&\Longleftrightarrow \quad |z_1|^\gamma < |z_2|  < 1 \qquad \Longleftrightarrow \quad z \in \h_\gamma.
\end{align*}
Since $\h_{\gamma}^{(2)} = \h_\gamma$, the holomorphic factor $k_\gamma$ is in fact a function on the original domain.
\end{proof}

\begin{theorem}\label{T:gen-Hartogs-biholo}
Let $\gamma,\eta > 0$.
The domains $\h_\gamma$ and $\h_\eta$ are biholomorphic if and only if 
\begin{equation}\label{E:nec-and-suff-cond-biholo}
\frac{1}{\eta} - \frac{1}{\gamma} := b \in \Z.
\end{equation}
In this case, the map $\Phi_{1,b}(z) = (z_1 z_2^b, z_2)$ is a biholomorphism $\h_\gamma \to \h_\eta$.
\end{theorem}
\begin{proof}
Let $\gamma,\eta > 0$ and suppose that $\h_\gamma$ and $\h_\eta$ are biholomorphic.
Since these domains are Reinhardt, Proposition \ref{P:biholo=algequiv} above says they are algebraically equivalent in the sense of Definition \ref{D:def-of-algebraic-equiv}.
Thus, there is a biholomorphism $\Phi:\h_\gamma \to \h_\eta$ of the form
\begin{align}
w = \Phi(z) &= (\Phi_1(z),\Phi_2(z)) \notag \\
&= (Az_1^a z_2^b, B z_1^c z_2^d)
= (w_1,w_2), \label{E:form1-monomial-map}
\end{align}
with $A,B \in \C^*$ and $a,b,c,d \in \Z$ with $ad-bc \neq 0$.
We now write
\begin{align*}
\h_\gamma = \{|z_1|^\gamma < |z_2| < 1 \}, \qquad
\h_\eta = \{|w_1|^\eta < |w_2| < 1 \},
\end{align*}
and turn to finding conditions on the parameters $A,B,a,b,c,d$ in \eqref{E:form1-monomial-map}.

Since points $z = (z_1,z_2) \in \h_\gamma$ are allowed to have $z_1=0$, this forces $a \ge 0$ and $c \ge 0$.

In fact, we must have $a>0$ because no point $(z_1,z_2) \in \h_\gamma$ has $z_2 = 0$ in its second coordinate, and therefore the only possible way to use $\Phi$ to map onto a point $(0,w_2) \in \h_\eta$ from a point $(z_1,z_2) \in \h_\gamma$ would be for $z_1 = 0$ and $a>0$.

Now notice that if $c>0$, then every point of the form $(0,z_2)$ would be mapped by $\Phi$ to $(0,0) \notin \h_\eta$.
We thus conclude $c = 0$, and therefore $ad \neq 0$ from the non-vanishing determinant condition above. 
This means
\begin{equation}\label{E:form2-monomial-map}
\Phi(z) = (Az_1^a z_2^b, B z_2^d),
\end{equation}
with $A,B \in \C^*$, and $a,b,d \in \Z$ with $a>0$ and $d \neq 0$.

We can also conclude that $d>0$ as follows.
Since $\h_\eta \subset \D \times \D^*$, the second component of $\Phi$ must satisfy
\begin{equation}\label{E:second-comp-ineq}
|\Phi_2(z)| = |B z_2^d| < 1.
\end{equation}
But since we are allowed to choose points $(z_1,z_2) \in \h_\gamma$ with $z_2 \to 0$, the bound in \eqref{E:second-comp-ineq} is only possible for $d>0$.

Now observe that $z_1$ only appears in the first component of the map \eqref{E:form2-monomial-map}.
This forces $a=1$, since otherwise all points $(w_1,w_2) \in \h_\eta$ with $w_1 \neq 0$ would have multiple pre-images under $\Phi$.
The same line of thinking also shows that $d=1$, since, again, points in $\h_\eta$ would otherwise have multiple pre-images.
This now means
\begin{equation}\label{E:form3-monomial-map}
\Phi(z) = (Az_1 z_2^b, B z_2),
\end{equation}
with $A,B \in \C^*$, and $b \in \Z$.


Since $\Phi : \h_\gamma \to \h_\eta$ is a biholomorphism, $z \in \h_\gamma$ if and only if $w = \Phi(z) \in \h_\eta$, which is to say
\begin{equation}\label{E:biholo-form1} 
|w_1|^\eta < |w_2| < 1 \qquad \Longleftrightarrow \qquad |z_1|^\gamma < |z_2| < 1. 
\end{equation}

On the other hand, \eqref{E:form3-monomial-map} shows
\begin{align}
|w_1|^\eta < |w_2| < 1 \qquad &\Longleftrightarrow \qquad |w_1|^\eta < |w_2| \notag \quad \mathrm{and} \quad 0 < |w_2| < 1, \notag \\
&\Longleftrightarrow \qquad |A z_1 z_2^b|^\eta < |B z_2| \quad \mathrm{and} \quad 0 < |B z_2| < 1, \notag \\
&\Longleftrightarrow \qquad |B|^{-1} |A z_1|^\eta < |z_2|^{1-b\eta} \notag \quad \mathrm{and} \quad 0 < |z_2| < |B|^{-1}, \notag \\
&\Longleftrightarrow \qquad |B|^{1/(b\eta-1)} |A z_1|^{1 / {(\frac{1}{\eta}-b})} < |z_2| \quad \mathrm{and} \quad 0 < |z_2| < |B|^{-1}. \label{E:biholo-form2}
\end{align}
Now comparing \eqref{E:biholo-form1} and \eqref{E:biholo-form2}, we obtain two equivalent descriptions for the source domain
\begin{align}
\h_\gamma &= \{ (z_1,z_2) : |z_1|^\gamma < |z_2| < 1 \} \notag \\
&=  \left\{ (z_1,z_2) : |B|^{1/(b\eta-1)} | A z_1|^{1 / {(\frac{1}{\eta}-b})} < |z_2| < |B|^{-1} \right\}. 
\end{align}
Thus, we see that this is possible if and only if the following three conditions hold
\begin{align}
|B|^{1/(b\eta-1)} |A|^{1 / {(\frac{1}{\eta}-b})} = 1, \qquad 1/(\tfrac{1}{\eta}-b) = \gamma, \qquad |B|^{-1} = 1.
\end{align}
This shows that there are $\theta_1,\theta_2 \in \R$ so that $A = e^{i\theta_1}, B = e^{i\theta_2}$ and 
\begin{equation*}
\frac{1}{\eta} - \frac{1}{\gamma} = b.
\end{equation*}
Since $b$ is an integer, the last displayed equation gives the necessary and sufficient condition in \eqref{E:nec-and-suff-cond-biholo}.
We see that when $\gamma,\eta > 0$ satisfy this, every algebric automorphism of the form $\Phi(z) = (e^{i\theta_1} z_1 z_2^b, e^{i\theta_2} z_2)$ restricts to a biholomorphism of $\h_\gamma \to \h_\eta$.
Setting $\theta_1 = \theta_2 = 0$ gives the map $\Phi_{1,b}$, as claimed.
\end{proof}

See Zapa\l owski \cite{Zapalowski17} for related results on proper holomorphic maps between higher dimensional variants of the Hartogs triangle.

We now see that each biholomorphic equivalence class in the domain family $\{\h_\gamma : \gamma > 0 \}$ contains a unique representative $\h_{\wt\gamma}$ with $\wt\gamma \ge 1$.

\begin{corollary}\label{C:biholo-equiv-classes}
Given $\gamma, \eta \ge 1$ with $\gamma \neq \eta$, the domains $\h_\gamma$ and $\h_\eta$ are biholmorphically inequivalent.
But given $\gamma \in (0,1)$, there is a unique $\wt{\gamma} \ge 1$ so that $\h_\gamma$ is biholomorphic to $\h_{\wt\gamma}$.
\end{corollary}
\begin{proof}
Suppose without loss of generality that $\gamma,\eta \ge 1$, $\gamma \neq \eta$, satisfy 
$1\le\eta<\gamma$. 
Then
\begin{equation}
\frac{1}{\gamma} < \frac{1}{\eta} \le 1 \qquad \Longrightarrow \qquad 0 < \frac{1}{\eta} - \frac{1}{\gamma} \le 1 - \frac{1}{\gamma},
\end{equation}
which means $\frac{1}{\eta} - \frac{1}{\gamma}$ cannot be an integer, and thus by Theorem \ref{T:gen-Hartogs-biholo}, $\h_\gamma$ and $\h_\eta$ cannot be biholomorphic.

Now let $\gamma \in (0,1)$. 
If $\frac{1}{\gamma} \in \Z$, then $\h_\gamma$ is biholomorphic to $\h_1$. 
Indeed, set $\wt\gamma = 1$ so that
\begin{equation*}
\frac{1}{\gamma} - \frac{1}{\wt\gamma} \in \Z,
\end{equation*}
and now invoke  Theorem \ref{T:gen-Hartogs-biholo}.

If $\frac{1}{\gamma} \notin \Z$, let us define 
\begin{equation}
x := \frac{1}{\gamma} - \left\lfloor \frac{1}{\gamma} \right\rfloor \in (0,1),
\end{equation}
and set 
\begin{equation}
\wt\gamma := \frac{1}{x} > 1.
\end{equation}
It is now easy to see that
\begin{align}
\frac{1}{\gamma} - \frac{1}{\wt\gamma} = \frac{1}{\gamma} - x = \left\lfloor \frac{1}{\gamma} \right\rfloor \in \Z.
\end{align}
Theorem \ref{T:gen-Hartogs-biholo} again gives the result.
\end{proof}


\section{An effective Bergman kernel formula}\label{S:effective-formula}

Let $m>n$ be relatively prime positive integers.
We begin with the numerator polynomial of the Bergman kernel of $\h_{m/n}$ given in \eqref{E:Almugh-formula}:
\begin{equation}\label{E:Almugh-numerator}
\sum_{\beta\in \Z^2} D_m(\beta_1) D_m(m\beta_2+n\beta_2 + m + n -1 -2mn) s^{\beta_1} t^{\beta_2} := P_{m,n}(s,t).
\end{equation}
Our goal is to eliminate the $D_m$ function that was defined in \eqref{E:Dbasic_def}.

\subsection{Determining non-zero coefficients}

To condense the presentation of \eqref{E:Almugh-numerator}, we introduce the quantities
\begin{align}
&\kappa(\beta_1,\beta_2) := m\beta_2+n\beta_1 + m + n -1 -2mn, \label{E:def-of-kappa(b1,b2)} \\
&C(\beta_1,\beta_2) := D_m(\beta_1) D_m(\kappa(\beta_1,\beta_2)). \label{E:def-of-C(b1,b2)}
\end{align}

We now determine all pairs $(\beta_1,\beta_2) \in \Z^2$ such that $C(\beta_1,\beta_2) \neq 0$. 

\begin{proposition}\label{P:Finding-effective-coefficients}
Fix a rational number $\frac{m}{n}>1$, $\gcd(m,n)=1$, and let $C(\beta_1,\beta_2)$ be defined as in \eqref{E:def-of-C(b1,b2)}.
The complete list of pairs $(\beta_1,\beta_2) \in \Z^2$ for which $C(\beta_1,\beta_2) \neq 0$ is determined explicitly.

\begin{itemize}
\item[(0)] $C(\beta_1,\beta_2) \neq 0$ if and only if both $0 \le \beta_1 \le 2m-2$ and $0 \le \kappa(\beta_1,\beta_2) \le 2m-2$.
\end{itemize}

\noindent  Case 1: $\beta_1 = m-1$ or $\kappa(\beta_1,\beta_2) = m-1$.

\begin{itemize}
\item[(1a)] If $\beta_1 = m-1$, the only possible $\beta_2$ value for which $C(\beta_1,\beta_2) \neq 0$ is $\beta_2 = n$, in which case $\kappa(\beta_1,\beta_2) = m-1$.

\item[(1b)] Suppose that $\beta_1,\beta_2$ are chosen so that $C(\beta_1,\beta_2) \neq 0$ and $\kappa(\beta_1,\beta_2) = m-1$.
Then $\beta_1 = m-1$ and $\beta_2 = n$.
\end{itemize}

In what follows, define 
\begin{equation}\label{E:def-of-mn-ceiling-func}
L(j) = \left\lceil \frac{1+n(j+1)}{m} \right\rceil.
\end{equation}

\noindent  Case 2: Let $\beta_1 = j \in \{0,1,\cdots,m-2 \}$ be fixed.
\begin{itemize}

\item[(2a)] 
There is a unique $\beta_2(j) \in \Z$ so that $\kappa(j,\beta_2(j)) \in \{0,1,\cdots,m-2 \}$.

Explicitly, 
\begin{equation}
\beta_2(j) = 2n - L(j), \qquad \kappa(j,\beta_2(j)) = m + n -1 + nj - m L(j),
\end{equation}
so each $\beta_2(j) \in \{n,n+1, \cdots, 2n-1 \}$.
Distinct choices of $j_1,j_2 \in \{0,1,\cdots,m-2 \}$ yield $\kappa(j_1,\beta_2(j_1)) \neq \kappa(j_2,\beta_2(j_2))$, meaning that $\kappa(j,\beta_2(j))$ assumes each value in $\{0,1,\cdots,m-2 \}$ exactly once as $j$ ranges in $\{0,1,\cdots,m-2 \}$.

\item[(2b)] 
There is a unique $\tilde{\beta}_2(j) \in \Z$ so that $\kappa(j,\tilde\beta_2(j)) \in \{m,m+1,\cdots, 2m-2 \}$.

Explicitly, 
\begin{alignat}{2}
\tilde\beta_2(j) &= 2n + 1 - L(j), \qquad \kappa(j,\tilde\beta_2(j)) &&= 2m + n -1 + nj - m L(j), \\
 &= \beta_2(j) + 1, \qquad &&= \kappa(j,\beta_2(j)) + m. \notag
\end{alignat}
so each $\tilde\beta_2(j) \in \{n+1,n+2, \cdots, 2n \}$.
Distinct choices of $j_1,j_2 \in \{0,1,\cdots,m-2 \}$ yield $\kappa(j_1,\tilde\beta_2(j_1)) \neq \kappa(j_2,\tilde\beta_2(j_2))$, meaning that $\kappa(j,\tilde\beta_2(j))$ assumes each value in $\{m,m+1,\cdots,2m-2 \}$ exactly once as $j$ ranges in $\{0,1,\cdots,m-2 \}$.
\end{itemize}

\noindent  Case 3: Let $\beta_1 = j \in \{m,m+1,\cdots,2m-2 \}$ be fixed.

\begin{itemize}

\item[(3a)] 
There is a unique $\beta_2(j) \in \Z$ so that $\kappa(j,\beta_2(j)) \in \{0,1,\cdots,m-2 \}$.

Explicitly, 
\begin{equation}
\beta_2(j) = 2n - L(j), \qquad \kappa(j,\beta_2(j)) = m + n -1 + nj - m L(j),
\end{equation}
so each $\beta_2(j) \in \{0,1, \cdots, n-1 \}$.
Distinct choices of $j_1,j_2 \in \{m,m+1,\cdots,2m-2 \}$ yield $\kappa(j_1,\beta_2(j_1)) \neq \kappa(j_2,\beta_2(j_2))$, meaning that $\kappa(j,\beta_2(j))$ assumes each value in $\{0,1,\cdots,m-2 \}$ exactly once as $j$ ranges in $\{m,m+1,\cdots,2m-2 \}$.

\item[(3b)] 
There is a unique $\tilde{\beta}_2(j) \in \Z$ so that $\kappa(j,\tilde\beta_2(j)) \in \{m,m+1,\cdots, 2m-2 \}$.

Explicitly, 
\begin{alignat}{2}
\tilde\beta_2(j) &= 2n + 1 - L(j), \qquad \kappa(j,\tilde\beta_2(j)) &&= 2m + n -1 + nj - m L(j), \\
 &= \beta_2(j) + 1, \qquad &&= \kappa(j,\beta_2(j)) + m. \notag
\end{alignat}
so each $\tilde\beta_2(j) \in \{1,2, \cdots, n \}$.
Distinct choices of $j_1,j_2 \in \{m,m+1,\cdots,2m-2 \}$ yield $\kappa(j_1,\tilde\beta_2(j_1)) \neq \kappa(j_2,\tilde\beta_2(j_2))$, meaning that $\kappa(j,\tilde\beta_2(j))$ assumes each value in $\{m,m+1,\cdots,2m-2 \}$ exactly once as $j$ ranges in $\{m,m+1,\cdots,2m-2 \}$.
\end{itemize}
\begin{proof}
Statement $(0)$ is clear from the definition of $D_m(\beta)$ given in \eqref{E:Dbasic_def}, and it sets the stage for the three subsequent cases.

Case $1$ concerns the situation when either $\beta_1$ or $\kappa(\beta_1,\beta_2)$ is equal to $m-1$, that is, the input value where the function $D_m$ peaks.

For $(1a)$, let $\beta_1 = m-1$. 
If $\beta_2 \in \Z$ is chosen so that $0 \le \kappa(m-1,\beta_2) \le 2m-2$ (which is required for $C(\beta_1,\beta_2)$ to be non-zero), definition \eqref{E:def-of-kappa(b1,b2)} shows
\begin{equation*}
n-1+\frac{1}{m} \le \beta_2 \le n + 1 -\frac{1}{m}.
\end{equation*}
This forces $\beta_2 = n$, and a computation shows $\kappa(m-1,n) = m-1$.

For $(1b)$, assume $\beta_1,\beta_2$ are chosen so that $\kappa(\beta_1,\beta_2) = m-1$.
This means
\begin{equation*}
m(\beta_2-2n) = -n(\beta_1+1) \qquad \Longrightarrow \qquad \beta_1 \equiv -1 \mod m,
\end{equation*}
since $\gcd(m,n)=1$.
Since $0\le \beta_1 \le 2m-2$ (which is required for $C(\beta_1,\beta_2)$ to be non-zero), this forces $\beta_1 = m-1$, which again implies that $\beta_2 = n$.
This concludes case $1$.

In Case 2, we let $\beta_1 = j \in \{0,1,\cdots,m-2 \}$, i.e., the inputs where $D_m$ is ascending, up to (but not including) the peak input at $j=m-1$.

For $(2a)$, assume that $\beta_2$ is chosen so that $0 \le \kappa(j,\beta_2) \le m-2$.
This means 
\begin{equation}\label{E:Case2_beta2_interval_1}
2mn +1 -m -n -nj \le m \beta_2 \le 2mn -1 -n -nj
\end{equation}
In the case $m=2$ (which forces $n=1$), we immediately have $j=0$, $\beta_2 = 1$ and $\kappa(0,1) = 0$.
For integers $m \ge 3$, \eqref{E:Case2_beta2_interval_1} is a closed interval of length $m-2$ with integer endpoints (meaning it contains $m-1$ integers).
We claim this interval must contain a multiple of $m$.
If not, the left endpoint would have to satisfy
\begin{equation*}
2mn +1 -m -n -nj \equiv 1 \mod m.
\end{equation*}
But this would then necessitate
\begin{equation*}
-n(j+1) \equiv 0 \mod m \qquad \Longrightarrow \qquad j \equiv -1 \mod m,
\end{equation*}
which is impossible since $j \in \{0,1 \cdots, m-2\}$.
Thus, for each $j$ in this range, there is a unique $\beta_2 = \beta_2(j) \in \Z$ satisfying \eqref{E:Case2_beta2_interval_1}, given by
\begin{equation}\label{E:beta2-def-inside-proof}
\beta_2(j) = \left\lfloor 2n - \frac{1+n(j+1)}{m}  \right\rfloor = 2n - L(j),
\end{equation}
where $L(j)$ is defined by \eqref{E:def-of-mn-ceiling-func}. 
Since $\frac{n}{m}<1$, $L(j)$ is a non-decreasing function of $j$.
Inspection shows that $L(0) = 1$ and $L(m-2) = n$, with each integer value between $1$ and $n$ attained at least once.
This means that $\beta_2(0) = 2n-1$ and that $\beta_2(m-2) = n$, with every integer value between attained by some $\beta_2(j)$ at least once.

Inserting \eqref{E:beta2-def-inside-proof} into \eqref{E:def-of-kappa(b1,b2)} shows 
\begin{equation*}
\kappa(j,\beta_2(j)) = m+n-1+nj-mL(j).
\end{equation*}
Now if $j_1,j_2 \in \{0,1,\cdots,m-2\}$ are chosen so that $\kappa(j_1,\beta_2(j_1)) = \kappa(j_2,\beta_2(j_2))$, we must have that
\begin{equation*}
n(j_1-j_2) = m(L(j_1)-L(j_2)) \qquad \Longrightarrow \qquad j_1 \equiv j_2 \mod m,
\end{equation*}
which is only possible if $j_1=j_2$.
We thus conclude that $\kappa(j,\beta_2(j))$ attains each value in $\{0,1,\cdots,m-2 \}$ exactly once as $j$ varies throughout the same range.

Part $(2b)$ is easily deduced from $(2a)$. 
Take $\tilde\beta_2(j) = \beta_2(j)+1$, where $\beta_2(j)$ is given by \eqref{E:beta2-def-inside-proof}, and see from \eqref{E:def-of-kappa(b1,b2)} that
\begin{align*}
\kappa(j,\tilde\beta_2(j)) &= m \tilde\beta_2(j) + nj + m + n - 1 - 2mn \\
&= m + m\beta_2(j) + nj + m + n - 1 - 2mn \\
&= m + \kappa(j,\beta_2(j)).
\end{align*}
Now since $\{m,m+1,\cdots,2m-2 \}$ is just a shift of $\{0,1,\cdots,m-2 \}$ by $m$ units, all claims made in $(2b)$ are equivalent to those in $(2a)$.
This concludes case 2.

Case 3 is handled similarly, but now $\beta_1 = j$ ranges in $\{m,m+1,\cdots,2m-2\}$, i.e., the range of values after the peak input at $j=m-1$ on which $D_m$ is descending.

For $(3a)$, assume $\beta_2$ is chosen so that $0 \le \kappa(j,\beta_2) \le m-2$.
This yields the same inequalities that appeared in \eqref{E:Case2_beta2_interval_1}, only now with a different range of $j$ values.
When $m=2$ (which forces $n=1$), we immediately have that $j=2$, so $\beta_2 = 0$ and $\kappa(2,0) = 0$.
For integers $m \ge 3$, \eqref{E:Case2_beta2_interval_1} is a closed interval of length $m-2$ with integer endpoints.
Following reasoning identical to that given above for $(2a)$, we see that for each $j \in \{m,m+1,\cdots,2m-2\}$, there is a unique $\beta_2 = \beta_2(j) \in \Z$ satisfying \eqref{E:Case2_beta2_interval_1}.
Once again $\beta_2(j) = 2n - L(j)$, the same formula provided in \eqref{E:beta2-def-inside-proof} for $j \in  \{0,1,\cdots,m-2 \}$.
Inspection shows that $\beta_2(m) = n-1$, and that $\beta_2(2m-2) = 0$, and reasoning identical to that given above shows that $\beta_2(j)$ attains each integer between $n-1$ and $0$ at least once.

Directly evaluating in \eqref{E:def-of-kappa(b1,b2)} again shows that $\kappa(j,\beta_2(j)) = m+n-1+nj-mL(j)$, and identical reasoning to that given above shows that if $\kappa(j_1,\beta_2(j_1)) = \kappa(j_2,\beta_2(j_2))$ for $j_1,j_2 \in \{m,m+1,\cdots,2m-2 \}$, then $j_1 = j_2$.
We thus conclude that $\kappa(j,\beta_2(j))$ attains each value in $\{0,1,\cdots,m-2 \}$ exactly once as $j$ varies throughout $\{m,m+1,\cdots,2m-2 \}$.

Part $(3b)$ is easily deduced from $(3a)$ in a way that parallels how $(2b)$ was deduced from $(2a)$.
We once again set $\tilde\beta_2(j) = \beta_2(j)+1$, which is seen to imply that $\kappa(j,\tilde\beta_2(j)) = m + \kappa(j,\beta_2(j))$.
As above, all claims made in $(3b)$ are equivalent to those made in $(3a)$.
This concludes case 3 and completes the proof.
\end{proof}
\end{proposition}

\subsection{Explicit determination of $P_{m,n}(s,t)$}\label{SS:explicit-determination-of-Pmn}

For the remainder of the paper, the quantities $\kappa(j,\beta_2(j))$ and $\kappa(j,\tilde\beta_2(j))$ will be abbreviated as follows:
\begin{subequations}
\begin{align}
\kappa(j) &:= \kappa(j,\beta_2(j)) = m + n - 1 +nj - mL(j), \label{def-of-kappaj} \\
\tilde\kappa(j) &:= \kappa(j,\tilde\beta_2(j)) = 2m + n - 1 +nj - mL(j) = m + \kappa(j). \label{def-of-kappatildej}
\end{align}
\end{subequations}

We first record some properties of $L(j)$ and $\kappa(j)$ that will be crucial later.

\begin{lemma}\label{L:props-of-L-and-kappa}
The functions $L$ and $\kappa$, defined in \eqref{E:def-of-mn-ceiling-func} and \eqref{def-of-kappaj}, respectively, satisfy the following functional identities
\begin{enumerate}
\item $L(j+m) = L(j) + n$.
\item $\kappa(j+m) = \kappa(j)$.
\item For $j \in \{0,1,\cdots,m-2 \}$, $L(j) + L(m-2-j) = n+1$.
\item For $j \in \{0,1,\cdots,m-2 \}$, $\kappa(j) + \kappa(m-2-j) = m-2$.
\end{enumerate}
\end{lemma}
\begin{proof}

(1) and (2) follow from trivial computation.
For (3), observe that (1) implies 
\[
L(j) + L(m-2-j) = n + L(j) + L(-2-j).
\]
By definition, $L(j) + L(-2-j) =$
\begin{equation}\label{E:int-step2-sum-of-Ls}
=\left\lceil \frac{1+n(j+1)}{m} \right\rceil + \left\lceil \frac{1-n(j+1)}{m} \right\rceil
= \left\lceil \frac{n(j+1)+1}{m} \right\rceil -\left\lfloor \frac{n(j+1) - 1}{m} \right\rfloor.
\end{equation}

For $j \in \{0,1,\cdots,m-2 \}$, write $n(j+1) = q(j) m + r(j)$, where $q(j) \in \Z$ is the quotient and $r(j)$ is the remainder. 
We claim $1 \le r(j) \le m-1$ for this range of $j$.
Indeed, since $\gcd(m,n) = 1$, $r(j)$ assumes each value in the set $\{1,2,\cdots,m-1 \}$ exactly once. If not, we could find $0 \le j_1 < j_2 \le m-2$ with $r(j_1) = r(j_2)$, which would imply $n(j_2-j_1) = (q(j_2)-q(j_1))m \equiv 0 \mod m$. 
But this is impossible since $0 < j_2 - j_1 \le m-2$.
Thus,
\begin{equation*}
\frac{n(j+1)+1}{m} = q(j) + \frac{r(j)+1}{m}, \qquad \frac{n(j+1)-1}{m} = q(j) + \frac{r(j)-1}{m},
\end{equation*}
and so
\begin{align*}
\left\lceil \frac{n(j+1)+1}{m} \right\rceil = q(j) + 1, \qquad \left\lfloor \frac{n(j+1)-1}{m} \right\rfloor = q(j).
\end{align*}
Now by \eqref{E:int-step2-sum-of-Ls} we see that for $j \in \{0,1,\cdots,m-2 \}$, $L(-2-j) + L(j) = 1$, proving (3).

For (4), notice that
\begin{align}
\kappa(m-2-j) + \kappa(j) &= 2m + 2n - 2 + n(m-2-j) - m L(m-2-j) + nj - mL(j) \notag \\
&= 2m - 2 + nm - m(L(m-2-j) + L(j)) \notag \\ 
&= 2m - 2 + nm  - m(n+1) \label{E:int-step1-sum-of-Ls} = m-2,
\end{align}
completing the proof.
\end{proof}

We are ready to eliminate the $D_m$ function and produce a formula for the numerator of the Bergman kernel in terms of expressions appearing in Proposition \ref{P:Finding-effective-coefficients}.

\begin{proof}[Proof of Theorem \ref{T:Bergman-kernel-formula}]
We use Proposition \ref{P:Finding-effective-coefficients} to decompose and simplify the numerator polynomial of the Bergman kernel given in Theorem \ref{T:Almugh-formula}:
\begin{equation*}
P_{m,n}(s,t) = \sum_{\beta\in \Z^2} D_m(\beta_1) D_m(m\beta_2+n\beta_2 + m + n -1 -2mn) s^{\beta_1} t^{\beta_2}.
\end{equation*}

\begin{itemize}[wide]
    
\item Proposition \ref{P:Finding-effective-coefficients} parts (1a) and (1b) together show that the two possible options in case 1 ($\beta_1 = m-1$ or $\kappa(\beta_1,\beta_2) = m-1$) in fact both correspond to the same, uniquely determined $\beta$-value: $(\beta_1,\beta_2) = (m-1,n)$.
This corresponds to a single monomial term in the sum above, which we now give a name.
Recalling the definition of $D_m$ in \eqref{E:Dbasic_def}, define
\begin{align*}
p_0(s,t) &:= D_m(\beta_1) D_m( \kappa(\beta_1,\beta_2) ) s^{\beta_1} t^{\beta_2} 
\Big|_{(\beta_1,\beta_2) = (m-1,n)} \\
&=D_m(m-1) D_m(m-1) s^{m-1} t^n \\
&= m^2 s^{m-1} t^n.
\end{align*}

\item Part (2a) considers $\beta_1 = j \in \{0,1,\cdots,m-2 \}$ and shows the integer $\beta_2(j) = 2n - L(j)$ is uniquely determined so that $\kappa(j) := \kappa(j,\beta_2(j)) \in \{0,1,\cdots,m-2 \}$. 
Recall the definition of $D_m$ in \eqref{E:Dbasic_def}, and give a name to the polynomial corresponding to these indices:
\begin{align*}
p_1(s,t) 
&:= \sum_{j=0}^{m-2} D_m(j) D_m(\kappa(j)) s^j t^{\beta_2(j)} = \sum_{j=0}^{m-2} (j+1) (\kappa(j)+1) s^j t^{2n-L(j)}.
\end{align*}

\item Part (2b) considers $\beta_1 = j \in \{0,1,\cdots,m-2 \}$ and shows the integer $\tilde\beta_2(j) = 2n + 1 - L(j)$ is uniquely determined so that $\tilde\kappa(j) := \kappa(j,\tilde\beta_2(j)) = m + \kappa(j) \in \{m,m+1,\cdots,2m-2 \}$. 
Recall the definition of $D_m$ in \eqref{E:Dbasic_def}, and give a name to the polynomial corresponding to these indices:
\begin{align*}
p_2(s,t) 
&:= \sum_{j=0}^{m-2} D_m(j) D_m(m+\kappa(j)) s^j t^{\tilde\beta_2(j)} \\
&= \sum_{j=0}^{m-2} (j+1) (m-\kappa(j)-1) s^j t^{2n+1-L(j)}.
\end{align*}

\item Part (3a) considers $\beta_1 = j \in \{m,m+1,\cdots,2m-2 \}$ and shows the integer $\beta_2(j) = 2n - L(j)$ is uniquely determined so that $\kappa(j) := \kappa(j,\beta_2(j)) \in \{0,1,\cdots,m-2 \}$. 
Recall the definition of $D_m$ in \eqref{E:Dbasic_def}, and give a name to the polynomial corresponding to these indices:
\begin{align*}
p_3(s,t) 
&:= \sum_{j=m}^{2m-2} D_m(j) D_m(\kappa(j)) s^j t^{\beta_2(j)} \\
&= \sum_{j=m}^{2m-2} (2m-j-1) (\kappa(j)+1) s^j t^{2n-L(j)} = \sum_{j=0}^{m-2} (m-j-1) (\kappa(j)+1) s^{j+m} t^{n-L(j)}.
\end{align*}
In the last step, we re-indexed the sum using Lemma \ref{L:props-of-L-and-kappa} to see that $\kappa(j+m)=\kappa(j)$ and $L(j+m) = L(j)-n$.

\item Part (3b) considers $\beta_1 = j \in \{m,m+1,\cdots,2m-2 \}$ and shows the integer $\tilde\beta_2(j) = 2n + 1 - L(j)$ is uniquely determined so that $\tilde\kappa(j) := \kappa(j,\tilde\beta_2(j)) = m + \kappa(j) \in \{m,m+1,\cdots,2m-2 \}$. 
Recall the definition of $D_m$ in \eqref{E:Dbasic_def}, and give a name to the polynomial corresponding to these indices:
\begin{align*}
p_4(s,t) 
&:= \sum_{j=m}^{2m-2} D_m(j) D_m(m+\kappa(j)) s^j t^{\tilde\beta_2(j)} \\
&= \sum_{j=m}^{2m-2} (2m-j-1) (m-\kappa(j)-1) s^j t^{2n+1-L(j)} \\
&= \sum_{j=0}^{m-2} (m-j-1) (m-\kappa(j)-1) s^{j+m} t^{n+1-L(j)}.
\end{align*}
\end{itemize}
As above, we re-indexed the last sum using Lemma \ref{L:props-of-L-and-kappa}.
By Proposition \ref{P:Finding-effective-coefficients} we have exhausted all indices $(\beta_1,\beta_2)$ corresponding to non-zero coefficients in the sum defining $P_{m,n}(s,t)$.
The proof is now complete.
\end{proof}


\subsection{Palindromic polynomials}\label{SS:palindromic-polys}

In this section we restrict the numerator polynomial of the Bergman kernel to the variety $\{s = t\}$.
The function $s \mapsto P_{m,n}(s,s)$ has a root of order $2n-1$ at $s=0$ (this is shown below), which suggests the definition of the following polynomial
\begin{align}\label{E:def-of-Qmn}
Q_{m,n}(s) := s^{1-2n} P_{m,n}(s,s).
\end{align}

Observe that $Q_{m,n}$ has roots in the punctured unit disc $\D^* \subset \C$ if and only if $P_{m,n}$ has roots in the punctured disc $\Delta^* = \{(s,s): s \in \D^* \} \subset \h_{m/n}$.

The main result in this section is that $Q_{m,n}$ is a palindromic polynomial of degree $2m-2n$.
Recall that a degree $2k$ polynomial 
\begin{equation}\label{E:degree2k-poly}
f(s) = \sum_{j=0}^{2k} \alpha_j s^j
\end{equation}
is said to be palindromic when its coefficients satisfy $\alpha_j = \alpha_{2k-j}$ for $0\le j \le k$.

With $f$ defined as in \eqref{E:degree2k-poly}, it can be seen to be  palindromic if and only if it satisfies
\begin{equation}\label{E:Palindromic-condition}
f(s) = s^{2k} f(s^{-1}).
\end{equation}

Equation \eqref{E:Palindromic-condition} shows that $s \in \C\backslash \{0\}$ is a root of the palindromic polynomial $f$ if and only if $s^{-1}$ is also a root.
This immediately implies the following:

\begin{proposition}\label{P:Palindromic-roots}
A palindromic polynomial $f$ of degree $2k$ has an even number of roots (counting multiplicities) away from the unit circle, with precisely half belonging to $\D^*$.
In particular, $f$ has roots in $\D^*$ unless all its roots lie on the unit circle.
If $f$ is non-vanishing on the circle, it has exactly $k$ roots inside $\D^*$.
\end{proposition}

Mirroring the decomposition given in \eqref{E:Pmn-as-sum-of-5-terms}, we express $Q_{m,n}$ as the sum
\begin{equation}\label{E:Qmn-as-sum-of-5-terms}
Q_{m,n}(s) = q_0(s) + q_1(s) + q_2(s) + q_3(s) + q_4(s),
\end{equation}
where each 
\[
q_j(s) := s^{1-2n} p_j(s,s).
\]
The formulas for the polynomials $p_j$ given in Theorem \ref{T:Bergman-kernel-formula} thus yield
{\allowdisplaybreaks
\begin{subequations}\label{E:def-of-qjs}
\begin{align}
&q_0(s) = m^2 s^{m-n}, \label{E:def-of-q0} \\
&q_1(s) = \sum_{j=0}^{m-2} (j+1)(\kappa(j)+1) s^{j-L(j)+1}, \label{E:def-of-q1} \\
&q_2(s) = \sum_{j=0}^{m-2} (j+1)(m-\kappa(j)-1) s^{j-L(j)+2}, \label{E:def-of-q2} \\
&q_3(s) = \sum_{j=0}^{m-2} (m-j-1)(\kappa(j)+1) s^{j-L(j)+m-n+1}, \label{E:def-of-q3} \\
&q_4(s) = \sum_{j=0}^{m-2} (m-j-1)(m-\kappa(j)-1) s^{j-L(j)+m-n+2}. \label{E:def-of-q4} 
\end{align}
\end{subequations}
}

\begin{proposition}\label{P:Qmn-palindromic}
The $q_j$ are polynomials with positive coefficients, satisfying
\begin{subequations}\label{E:qj-functional-ids}
\begin{align}
&s^{2m-2n}q_0(s^{-1}) = q_0(s), \label{E:q0-q0}\\
&s^{2m-2n}q_1(s^{-1}) = q_4(s), \label{E:q1-q4}\\
&s^{2m-2n}q_2(s^{-1}) = q_3(s). \label{E:q2-q3}
\end{align}
Consequently, $Q_{m,n}$ is a palindromic polynomial of degree $2m-2n$ with positive coefficients.
\end{subequations}
\end{proposition}
\begin{proof}
Since $\frac{m}{n}>1$ is a rational number in lowest terms, we emphasize that $m \ge 2$.

Each coefficient inside the sums \eqref{E:def-of-q1} through \eqref{E:def-of-q4} is seen to be positive by recalling Proposition \ref{P:Finding-effective-coefficients} part $(2a)$, that $\kappa(j)$ assumes each value in $\{0,1,\cdots,m-2\}$ exactly once as $j$ ranges in $\{0,1,\cdots,m-2\}$. (Recall that in \eqref{def-of-kappaj} we simplified notation by writing $\kappa(j)$ for what was originally defined as $\kappa(j,\beta_2(j))$).

Next we show that each $q_j$ is a polynomial, i.e., that only non-negative powers of $s$ appear in each respective sum.
Clearly $q_0$ is a polynomial since we assume $m>n$.
Also notice that in the sums defining $q_1,q_2,q_3,q_4$, the exponent $j-L(j)+1$ appearing on $s$ in $q_1$ is smaller than the respective exponents of $s$ in $q_2,q_3,q_4$.
We write this exponent as
\begin{equation}\label{E:def-of-E(j)}
E(j) := j-L(j)+1.
\end{equation}
We now show this is non-negative for all $j = 0, 1, \cdots, m-2$.

First, we claim $E(j)$ is a non-decreasing function on the integers, i.e., $E(j+1) - E(j) \ge 0$. 
To see this, observe
\begin{align*}
E(j+1) - E(j) = 1 - \left( \left\lceil \frac{1+n(j+2)}{m} \right\rceil - \left\lceil \frac{1+n(j+1)}{m} \right\rceil \right)
\end{align*}
Notice that the difference of the two terms {\em inside the ceiling brackets} above is
\begin{align*}
\frac{1+n(j+2)}{m} - \frac{1+n(j+1)}{m} = \frac{n}{m} < 1,
\end{align*}
meaning that for any $j$, the difference in parentheses above takes one of two possible values:
\begin{align*}
\left\lceil \frac{1+n(j+2)}{m} \right\rceil - \left\lceil \frac{1+n(j+1)}{m} \right\rceil \in \{0,1\}.
\end{align*}
Thus for each $j$, $E(j+1) - E(j) \ge 0$.
To conclude that $E(j)$ is non-negative for $j \in  \{0,1,\cdots,m-2\}$, it is now sufficient to check the value at $j=0$:
\begin{align}
E(0) = -L(0) + 1 = -\left\lceil \frac{1+n}{m} \right\rceil + 1 = 0,
\end{align}
since $0 < 1+n \le m$.
This shows that $q_1$ is a polynomial with a non-zero constant term.
Consequently $q_2,q_3,q_4$ are also polynomials.

Now we verify the identities in \eqref{E:qj-functional-ids}.
First,
\begin{equation*}
s^{2m-2n}q_0(s^{-1}) = s^{2m-2n} m^2 s^{n-m} = m^2 s^{m-n} = q_0(s),
\end{equation*}
confirming \eqref{E:q0-q0}.
Next consider
\begin{align}
s^{2m-2n}q_1(s^{-1}) = \sum_{j=0}^{m-2} (j+1)(\kappa(j)+1) s^{2m-2n+L(j)-j-1}, \label{E:q1-id-step1}
\end{align}
and re-index the sum by setting $j = m-2-i$. 
By Lemma \ref{L:props-of-L-and-kappa}, we have that $\kappa(m-2-i) = m-2-\kappa(i)$ and $L(m-2-i) = n + 1 - L(i)$, and so
\begin{align*}
\eqref{E:q1-id-step1} 
&= \sum_{i=0}^{m-2} (m-1-i)(\kappa(m-2-i)+1) s^{2m-2n+L(m-2-i)-m+2+i-1} \\
&= \sum_{i=0}^{m-2} (m-1-i)(m-1-\kappa(i)) s^{i-L(i)+m-n+2} = q_4(s),
\end{align*}
confirming \eqref{E:q1-q4}.
Following the same line of reasoning, now consider
\begin{align*}
s^{2m-2n}q_2(s^{-1}) &= \sum_{j=0}^{m-2} (j+1)(m-\kappa(j)-1) s^{2m-2n+L(j)-j-2} \\
&= \sum_{i=0}^{m-2} (m-1-i)(m-\kappa(m-2-i)-1) s^{2m-2n+L(m-2-i)-m+2+i-2} \\
&= \sum_{i=0}^{m-2} (m-1-i)(\kappa(i)+1) s^{i-L(i)+m-n+1} = q_3(s),
\end{align*}
confirming \eqref{E:q2-q3}.
Finally, we show $Q_{m,n}(s)$ is palindromic of degree $2m-2n$:
\begin{align*}
s^{2m-2n}Q_{m,n}(s^{-1}) &= s^{2m-2n} \big( q_0(s^{-1}) +  q_1(s^{-1}) +  q_2(s^{-1}) +  q_3(s^{-1}) + q_4(s^{-1}) \big) \\
&= q_0(s) + q_4(s) + q_3(s) + q_2(s) + q_1(s) \\
&= Q_{m,n}(s),
\end{align*}
as \eqref{E:q1-q4} and \eqref{E:q2-q3} also imply $s^{2m-2n}q_4(s^{-1}) = q_1(s)$ and $s^{2m-2n}q_3(s^{-1}) = q_2(s)$.
\end{proof}


\section{The Lu Qi-Keng problem}\label{S:zeros-of-the-Bergman-kernel}

We use the work of the previous section to prove for infinitely many rationals $\gamma = \frac{m}{n} \in (1,2)$ the existence of zeros of the Bergman kernels on $\h_\gamma$.

\subsection{Case $m-n=1$}\label{SS:k=1}
$\big(\frac{m}{n} = \frac{2}{1}, \frac{3}{2}, \frac{4}{3}, \frac{5}{4}, \frac{6}{5}, \frac{7}{6}, \cdots \big)$.
 
\begin{proposition}\label{P:k=1-form-of-poly}
Let us write $(m,n) = (\ell+1, \ell)$, $\ell \in \Z^+$. 
The polynomial $Q_{\ell+1,\ell}$ takes the form
\begin{align*}
Q_{\ell+1,\ell}(s) &= \alpha_0(\ell) s^2 + \alpha_1(\ell) s + \alpha_0(\ell),
\end{align*}
where
\begin{align*}
\alpha_0(\ell) &= \tfrac{1}{6}\ell(\ell+1)(\ell+2), \\
\alpha_1(\ell) &= \tfrac{1}{3}(\ell+1)(2\ell^2+4\ell+3).
\end{align*}
\end{proposition}
\begin{proof}
For $(m,n) = (\ell+1,\ell)$ let us first calculate the quantities $L(j)$ and $\kappa(j)$ from \eqref{E:def-of-mn-ceiling-func-kappaj-intro}.
For the range of $j$ appearing in the sums defining $q_1,q_2,q_3,q_4$ $(0 \le j \le m-2 = \ell - 1)$, we have
\begin{align}\label{E:L(j)-k=1}
L(j) = \left\lceil \frac{1+\ell(j+1)}{\ell+1} \right\rceil 
= j+1.
\end{align}
A quick calculation now shows $\kappa(j) = \ell - j - 1$ and thus equations \eqref{E:def-of-q0}--\eqref{E:def-of-q4} give
\begin{align*}
q_0(s) = (\ell+1)^2 s, \qquad q_1(s) = \sum_{j=0}^{\ell-1} (j+1)(\ell-j), \qquad q_2(s) = s \sum_{j=0}^{\ell-1} (j+1)^2
\end{align*}
\begin{align*}
q_3(s) = s \sum_{j=0}^{\ell-1} (\ell-j)^2, \qquad q_4(s) = s^{2} \sum_{j=0}^{\ell-1} (\ell-j)(j+1).
\end{align*}
Now add these polynomials (in the following computation, it is helpful to reverse the order of summation in $q_3$):
\begin{align*}
Q_{\ell+1,\ell}(s) 
&= q_0(s) + q_1(s) + q_2(s) + q_3(s) + q_4(s) \notag \\
&= (1 + s^{2}) \left( \sum_{j=0}^{\ell-1} (j+1)(\ell-j) \right) + s\left( 2\sum_{j=0}^{\ell-1} (j+1)^2 + (\ell+1)^2 \right) \notag \\
&= (1 + s^{2}) \alpha_0(\ell) + s\alpha_1(\ell),
\end{align*}
where the formulas for $\alpha_0(\ell)$ and $\alpha_1(\ell)$ as cubics in $\ell$ follow from well known sums of powers formulas (see e.g. \cite{GrKuPa94}).
\end{proof}

\begin{proposition}\label{P:k=1-unique-root}
$Q_{\ell+1,\ell}(s)$ has a unique root inside the punctured unit disc $\D^*$.
\end{proposition}
\begin{proof} 
We factor
\begin{align}
Q_{\ell+1,\ell}(s) 
&= \alpha_0(\ell) s^2 + \alpha_1(\ell) s + \alpha_0(\ell) \notag \\
&= \alpha_0(\ell) \big(s-r_1(\ell)\big) \big(s-r_2(\ell)\big), \label{E:Q-quadratic-factored}
\end{align}
where $r_1(\ell)$ and $r_2(\ell)$ denote the two roots.
Observe that
\begin{equation}
\alpha_1(\ell) = \tfrac{1}{3}(\ell+1)(2\ell^2+4\ell+3)
>
\tfrac{1}{3}(\ell+1)(\ell^2+2\ell) = 2\alpha_0(\ell),
\end{equation}
which shows the roots $r_1(\ell)$ and $r_2(\ell)$ are real-valued and distinct.
Also observe from \eqref{E:Q-quadratic-factored} that $r_1(\ell) \cdot r_2(\ell) = 1$, which implies the roots have the same sign.
Thus, one of them (say $r_1(\ell)$) lies inside of $\D^*$ while $r_2(\ell)$ lies outside of $\ol{\D}$.
(It is also clear that the roots must be negative, as the coefficient functions $\alpha_0(\ell)$ and $\alpha_1(\ell)$ are always positive.)
\end{proof}

\begin{remark}\label{R:k=1_Limits-of-roots}
We can also use the quadratic formula to find the above roots explicitly:
\begin{equation*}
r_1(\ell) = -\left( \frac{2\ell^2+4\ell+3}{\ell^2+2\ell} \right) + \sqrt{ \left(\frac{2\ell^2+4\ell+3}{\ell^2+2\ell} \right)^2 - 1 }.
\end{equation*}
Now observe that as $\ell \to \infty$, these roots accumulate at a point inside the punctured disc:
\[
\lim_{\ell \to \infty} r_1(\ell) = -2 + \sqrt{3} \in \D^*.
\]
\end{remark}

\begin{corollary}\label{C:LQK-k=1}
The Bergman kernel of $\h_{\frac{\ell+1}{\ell}}$ has zeros inside $\h_{\frac{\ell+1}{\ell}} \times \h_{\frac{\ell+1}{\ell}}$ for every $\ell \in \Z^+$.
\end{corollary}
\begin{proof}
Since $\h_{\frac{\ell+1}{\ell}}$ is Reinhardt, Corollary \ref{C:LQK-Reinhardt-shadow} shows that its Bergman kernel has zeros inside $\h_{\frac{\ell+1}{\ell}} \times \h_{\frac{\ell+1}{\ell}}$ if and only if the corresponding holomorphic factor $k_{\frac{\ell+1}{\ell}}(s,t)$ has zeros inside the Reinhardt power domain $\h_{\frac{\ell+1}{\ell}}^{(2)}$.
But now recall that Proposition \ref{P:gen-Hartogs-are-their-own-Reinhardt-powers} shows $\h_\gamma^{(2)} = \h_\gamma$ for any $\gamma > 0$.
From equation \eqref{E:Bergman-kernel-formula-intro},
\begin{equation*}
k_{\frac{\ell+1}{\ell}}(s,t) = \frac{1}{\pi^2(\ell+1)} \cdot \frac{P_{\ell+1,\ell}(s,t)}{(1-t^2)^2(t^{\ell}-s^{\ell+1})^2},
\end{equation*}
where the denominator is non-vanishing for $(s,t) \in \h_{\frac{\ell+1}{\ell}}$.
Proposition \ref{P:k=1-unique-root} shows that $Q_{\ell+1,\ell}$ has a unique root in $\D^*$, and \eqref{E:def-of-Qmn} shows that this root corresponds to a root of $P_{\ell+1,\ell}$ restricted to the punctured disc $\Delta^* = \{(s,s)\in\C^2 : s \in \D^* \} \subset \h_{\frac{\ell+1}{\ell}}$.
\end{proof}


\subsection{Case $m-n=2$}\label{SS:k=2} $\big(\frac{m}{n} = \frac{3}{1}, \frac{5}{3}, \frac{7}{5}, \frac{9}{7}, \frac{11}{9}, \frac{13}{11}, \cdots \big)$.

\begin{proposition}\label{P:k=2-form-of-poly}
Let $(m,n) = (2\ell+1, 2\ell-1)$, $\ell \in \Z^+$. 
The polynomial $Q_{2\ell+1,2\ell-1}$ takes the form
\begin{align*}
Q_{2\ell+1,2\ell-1}(s) &= \alpha_0(\ell) (1+ s^4) + \alpha_1(\ell) (s+s^3) + \alpha_2(\ell) s^2,
\end{align*}
where
\begin{align*}
\alpha_0(\ell) &= \tfrac{1}{6}\ell(\ell+1)(2\ell+1), \\
\alpha_1(\ell) &= \ell(\ell+1)(2\ell+1),\\
\alpha_2(\ell) &= \tfrac{1}{3}(2\ell+1)(5\ell^2+5\ell+3).
\end{align*}
\end{proposition}
\begin{proof}
We compute $L(j)$ and $\kappa(j)$ from \eqref{E:def-of-mn-ceiling-func-kappaj-intro}, and see that the relevant $j$ values satisfy $0 \le j \le m-2 = 2\ell - 1$.
\begin{align}
L(j) 
= \left\lceil \frac{1+(2\ell-1)(j+1)}{2\ell+1} \right\rceil
&= j+1 - \left\lfloor \frac{2(j+1)-1}{2\ell+1} \right\rfloor \notag \\
&=
\begin{cases} 
j + 1, & \quad  0 \le j \le \ell - 1,  \\ 
j, &  \quad \ell \le j \le 2\ell - 1.
\end{cases} \label{E:L(j)_k=2}
\end{align}
A short calculation now shows
\begin{equation}\label{E:kappa(j)_k=2}
\kappa(j) =
\begin{cases} 
2\ell - 2 - 2j, & \quad  0 \le j \le \ell - 1,  \\ 
4\ell - 1 - 2j, &  \quad \ell \le j \le 2\ell - 1.
\end{cases} 
\end{equation}

Equations \eqref{E:def-of-q0}--\eqref{E:def-of-q2} now give
{
\allowdisplaybreaks
\begin{align}
q_0(s) 
&= (2\ell+1)^2 s^2, \\
q_1(s)
&= \sum_{j=0}^{2\ell-1} (j+1)(\kappa(j)+1) s^{j-L(j)+1} \notag \\
&= \sum_{j=0}^{\ell-1} (j+1)(2\ell-1-2j) s^{j - (j+1) + 1} + \sum_{j=\ell}^{2\ell-1} (j+1)(4\ell-2j) s^{j - j + 1} \notag \\
&= \left( \frac{\ell(\ell+1)(2\ell+1)}{6} \right) + \left(\frac{2 \ell(\ell+1)(2\ell+1)}{3} \right) s, \label{E:q1_k=2} \\
q_2(s) 
&= \sum_{j=0}^{2\ell-1} (j+1)(2\ell -\kappa(j)) s^{j-L(j)+2} \notag \\
&= 2s \sum_{j=0}^{\ell-1} (j+1)^2 + s^2 \sum_{j=\ell}^{2\ell-1} (j+1)(2j+1-2\ell) \notag \\
&= \left(\frac{\ell(\ell+1)(2\ell+1)}{3}\right) s + \left( \frac{\ell(2\ell+1)(5\ell-1)}{6} \right) s^2. \notag
\end{align}
}

Once again, these sums are computed in closed form by using well known sums of powers formulas (see e.g. \cite{GrKuPa94}).
Proposition \ref{P:Qmn-palindromic} now shows
\begin{align}
q_3(s) &= \left(\frac{\ell(\ell+1)(2\ell+1)}{3}\right) s^3 + \left( \frac{\ell(2\ell+1)(5\ell-1)}{6} \right) s^2, \notag \\
q_4(s) &= \left( \frac{\ell(\ell+1)(2\ell+1)}{6} \right) s^4 + \left(\frac{2 \ell(\ell+1)(2\ell+1)}{3} \right) s^3. \notag
\end{align}
Now add $q_0+q_1+q_2+q_3+q_4$ to confirm $Q_{2\ell+1,2\ell-1}(s)$ takes the form given in the statement of the theorem.
\end{proof}

\begin{proposition}\label{P:k=2-double-root}
$Q_{2\ell+1,2\ell-1}(s)$ has exactly two roots inside the punctured unit disc $\D^*$.
\end{proposition}

\begin{proof}
To locate the roots of $Q_{2\ell+1,2\ell-1}(s)$, we start with the formula in Proposition \ref{P:k=2-form-of-poly} and factor out the coefficient of $(1+s^4)$:
\begin{align}
Q_{2\ell+1,2\ell-1}(s) 
&= \frac{\ell(\ell+1)(2\ell+1)}{6} \left( 1 + 6s + \left(10 + \frac{6}{\ell^2+\ell} \right) s^2 + 6s^3 + s^4 \right) \notag \\
&:= \frac{\ell(\ell+1)(2\ell+1)}{6} s^2 R_\ell(s), \label{E:Q_k=2_step1}
\end{align}
where
\begin{equation}\label{E:def-R_l(s)}
R_\ell(s) =  s^{-2} + 6 s^{-1} + \left(10 + \frac{6}{\ell^2+\ell} \right) + 6s + s^2.
\end{equation}

The roots of $Q_{2\ell+1,2\ell-1}$ are exactly the roots of $R_\ell(s)$. 
We now show that $R_\ell(s)$ is non-vanishing on the unit circle.

Suppose the contrary, that $s = e^{i\theta}$ is a root for some $\theta \in \R$. 
Then
\begin{equation}
R_\ell(e^{i\theta}) = 4\cos^2\theta + 12 \cos\theta  + \left(8 + \frac{6}{\ell^2+\ell} \right) = 0,
\end{equation}
which is equivalent to saying there exists $x = \cos\theta \in [-1,1]$ such that
\begin{equation*}
G_\ell(x) := 4x^2 + 12 x  + \beta(\ell) = 0, \qquad \mathrm{where} \qquad \beta(\ell) =  \left(8 + \frac{6}{\ell^2+\ell} \right).
\end{equation*}
The roots of $G_\ell(s)$ can be calculated explicitly:
\begin{equation}
r_1(\ell) = -\frac{3}{2} + \frac{1}{2}\sqrt{9-\beta(\ell)}, \qquad 
r_2(\ell) = -\frac{3}{2} - \frac{1}{2}\sqrt{9-\beta(\ell)}.
\end{equation}
Note that $\ell \mapsto \beta(\ell)$ is a decreasing function with $\beta(1) = 11$, $\beta(2)=9$ and $\lim_{\ell \to \infty} \beta(\ell) = 8$.

When $\ell = 1$, the roots above are the complex numbers:
\begin{equation*}
r_1(1) = -\frac{3}{2} + i\frac{\sqrt{2}}{2}, \qquad r_2(1) = -\frac{3}{2} - i\frac{\sqrt{2}}{2},
\end{equation*}
which are clearly not in the interval $[-1,1]$.
For $\ell \ge 2$, the roots are real numbers given by
\begin{equation*}
r_1(\ell) \in \left[ -\frac{3}{2},-1 \right), \qquad r_2(\ell) \in \left(-2,-\frac{3}{2} \right].
\end{equation*}
In particular, $G_\ell(x)$ has no roots in the interval $[-1,1]$, which means that $R_\ell(s)$, and consequently $Q_{2\ell+1,2\ell-1}(s)$, are non-vanishing on the unit circle.
Since $Q_{2\ell+1,2\ell-1}(s)$ is palindromic, Proposition \ref{P:Palindromic-roots} shows that exactly two of its roots lie inside $\D^*$.
\end{proof}

\begin{remark}\label{R:k=2_Limits-of-roots}
We now describe the limit behavior of the roots of $Q_{2\ell+1,2\ell-1}(s)$.
Note that by \eqref{E:Q_k=2_step1}, $s^2 R_\ell(s)$ is a monic polynomial sharing the same roots as $Q_{2\ell+1,2\ell-1}(s)$.
Now observe
\begin{align*}
\lim_{\ell \to \infty} s^2R_\ell(s) &= 1 + 6s + 10s^2 + 6s^3 + s^4 = (s+1)^2(s^2+4s+1),
\end{align*}
which has roots at $\{-1, -1, -2+\sqrt{3}, -2-\sqrt{3}\}$.
Thus one interior root moves to the unit circle ($s=-1$) as $\ell \to \infty$, but the other interior root stays inside $\D^*$ and approaches $s = -2 + \sqrt{3}$. 
In Remark \ref{R:k=1_Limits-of-roots} we showed this interior root accumulation point is also the root accumulation point when $m-n=1$.
\end{remark}

\begin{corollary}\label{C:LQK-k=2}
The Bergman kernel of $\h_{\frac{2\ell+1}{2\ell-1}}$ has zeros inside $\h_{\frac{2\ell+1}{2\ell-1}} \times \h_{\frac{2\ell+1}{2\ell-1}}$.
\end{corollary}
\begin{proof}
The same proof given in Corollary \ref{C:LQK-k=1} applies here with trivial modifications.
\end{proof}

\subsection{Case $m-n \ge 3$}\label{SS:k_ge_3}

We now briefly consider relatively prime pairs $m,n \in \Z^+$ with $m-n=k \ge 3$.
Though analysis becomes harder as $k$ grows, there is ample evidence to suggest that every $Q_{m,n}$ is non-vanishing on the unit circle.

To supplement the $k=1$ and $k=2$ cases in Sections \ref{SS:k=1} and \ref{SS:k=2}, the authors have also recently carried out arguments analogous to those given in Propositions \ref{P:k=1-unique-root} \and \ref{P:k=2-double-root}, proving the non-vanishing of $Q_{m,n}$ on the unit circle when $k = 3,4,6$.
These new arguments require analysis of polynomial families of degree $6, 8, 12$, respectively.
Interestingly, the (still open) $k=5$ case is more difficult because (letting $\varphi$ denote Euler's totient function; see \cite{GrKuPa94}) $\varphi(5) = 4$, while $\varphi(3) = 2$, $\varphi(4) = 2$,  $\varphi(6) = 2$.
The higher the value of $\varphi(k)$, the more separate cases must be considered, thus complicating the analysis.

The problem that seeks to locate the roots of $Q_{m,n}$ for a general pair of relatively prime integers $m > n$ is wide open.
The authors have, however, been able to prove that for each $k \in \Z^+$, there are an infinite number of explicitly given pairs $m > n$ with $m-n=k$ such that $Q_{m,n}$ is non-vanishing on the unit circle.
This argument requires detailed analysis of the polynomial coefficients considered in tandem with conditions bounding the number of real roots of cosine polynomials with real coefficients.
This material is the subject of a new paper by the authors currently in preparation.

In addition to these theoretical considerations, the authors have also used Mathematica to analyze an abundance of $Q_{m,n}$ polynomials outside the purview of the above results and have still {\em never observed a single one with a root on the unit circle}.

\begin{conjecture}\label{Conj:non-vanishing}
Let $m > n$ be positive integers with $\gcd(m,n)=1$.
The degree $2m-2n$ palindromic polynomial $Q_{m,n}$ is non-vanishing on the unit circle.
Consequently, $m-n$ of its roots lie inside $\D^*$.
\end{conjecture}

If this conjecture holds it would be somewhat remarkable, as palindromic polynomials are generically expected to have roots on the unit circle.
Indeed, a beautiful result of Dunnage \cite{Dunn1966} from 1966 shows that, given a single randomly generated palindromic polynomial with real coefficients, the expected value of the percentage of its zeros lying on the unit circle is $\frac{1}{\sqrt{3}}$, or about 57.7\%.

If Conjecture \ref{Conj:non-vanishing} can be proved, it would then immediately imply that the Bergman kernel of $\h_{m/n}$ has zeros inside $\h_{m/n} \times \h_{m/n}$ for every positive rational number $\frac{m}{n} \neq 1$.

\bibliographystyle{acm}
\bibliography{EdhMat24}
\end{document}